\documentclass[]{amsart}
\usepackage[english]{babel}
\usepackage[utf8x]{inputenc}
\usepackage[T1]{fontenc}

\usepackage[a4paper]{geometry}

\usepackage{amsmath,amscd}
\usepackage{amssymb}
\usepackage{graphicx}
\usepackage{verbatim}
\usepackage[colorinlistoftodos]{todonotes}
\usepackage[colorlinks=true, allcolors=blue]{hyperref}
\usepackage{tikz-cd}
\usepackage{mathtools}

\theoremstyle{plain}
\newtheorem{thm}{Theorem}[section]
\newtheorem{lemma}[thm]{Lemma}
\newtheorem{prop}[thm]{Proposition}
\newtheorem{cor}[thm]{Corollary}

\theoremstyle{definition}
\newtheorem{question}[thm]{Question}
\newtheorem{defn}[thm]{Definition}

\theoremstyle{remark}
\newtheorem{remark}[thm]{Remark}

\numberwithin{equation}{section}

\def\makeop#1{\expandafter\def\csname#1\endcsname
  {\mathop{\rm #1}\nolimits}\ignorespaces}
\makeop{Hom}   \makeop{End}   \makeop{Aut}   \makeop{Isom}  \makeop{Pic} 
\makeop{Gal}   \makeop{ord}   \makeop{Char}  \makeop{Div}   \makeop{Lie} 
\makeop{PGL}   \makeop{Corr}  \makeop{PSL}   \makeop{sgn}   \makeop{Spf}
\makeop{Spec}  \makeop{Tr}    \makeop{Nr}    \makeop{Fr}    \makeop{disc}
\makeop{Proj}  \makeop{supp}  \makeop{ker}   \makeop{im}    \makeop{dom}
\makeop{coker} \makeop{Stab}  \makeop{SO}    \makeop{SL}    \makeop{SL}
\makeop{Cl}    \makeop{cond}  \makeop{Br}    \makeop{inv}   \makeop{rank}
\makeop{id}    \makeop{Fil}   \makeop{Frac}  \makeop{GL}    \makeop{SU}
\makeop{Nrd}   \makeop{Sp}    \makeop{Tr}    \makeop{Trd}   \makeop{diag}
\makeop{Res}   \makeop{ind}   \makeop{depth} \makeop{Tr}    \makeop{st}
\makeop{Ad}    \makeop{Int}   \makeop{tr}    \makeop{Sym}   \makeop{can}
\makeop{length}\makeop{SO}    \makeop{torsion} \makeop{GSp} \makeop{Ker}
\makeop{Adm}   \makeop{Mat}  
\def\makebb#1{\expandafter\def
  \csname bb#1\endcsname{{\mathbb{#1}}}\ignorespaces}
\def\makebf#1{\expandafter\def\csname bf#1\endcsname{{\bf
      #1}}\ignorespaces} 
\def\makegr#1{\expandafter\def
  \csname gr#1\endcsname{{\mathfrak{#1}}}\ignorespaces}
\def\makescr#1{\expandafter\def
  \csname scr#1\endcsname{{\EuScript{#1}}}\ignorespaces}
\def\makecal#1{\expandafter\def\csname cal#1\endcsname{{\mathcal
      #1}}\ignorespaces} 

\def\doLetters#1{#1A #1B #1C #1D #1E #1F #1G #1H #1I #1J #1K #1L #1M
                 #1N #1O #1P #1Q #1R #1S #1T #1U #1V #1W #1X #1Y #1Z}
\def\doletters#1{#1a #1b #1c #1d #1e #1f #1g #1h #1i #1j #1k #1l #1m
                 #1n #1o #1p #1q #1r #1s #1t #1u #1v #1w #1x #1y #1z}
\doLetters\makebb   \doLetters\makecal  \doLetters\makebf
\doLetters\makescr 
\doletters\makebf   \doLetters\makegr   \doletters\makegr
     \def\qed{\qedmark\medbreak}%
\def\qedmark{{\enspace\vrule height 6pt width 5pt depth 1.5pt}}%
    
\def\Gm{{{\bbG}_{\rm m}}}

\def\Fpbar{\overline{\bbF}_p}
\def\Fp{{\bbF}_p}
\def\Fq{{\bbF}_q}

\def\Qp{{\bbQ}_p}

\def\Zp{{\bbZ}_p}
\def\Qbar{\overline{\bbQ}}
\def\Sh{{\rm Sh}}
\def\wh{\widehat}
\def\wt{\widetilde}

\def\Gm{\mathbb{G}_{\rm m}} 
\def\G{\mathbb{G}}
\def\R{\mathbb{R}}
\def\Q{\mathbb{Q}}
\def\Z{\mathbb{Z}}
\def\A{\mathbb{A}}
\def\C{\mathbb{C}}
\def\o{\mathfrak{o}}

\def\X{\times}
\def\Sq{{\rm Sq}}

\def\char{\text{char }}

\def\embed{\hookrightarrow}
\def\F{\bbF}
\def\ol{\overline}
\def\ul{\underline}

\newcommand{\npr}{\noindent }

\newcommand{\<}{\langle}   
\renewcommand{\>}{\rangle} 

\newcommand{\isoto}{\stackrel{\sim}{\to}}
\def\red#1{{\color{black} #1}}

\makeatletter
\newcommand{\xdashrightarrow}[2][]{\ext@arrow 0359\rightarrowfill@@{#1}{#2}}
\newcommand{\xdashleftarrow}[2][]{\ext@arrow 3095\leftarrowfill@@{#1}{#2}}
\newcommand{\xdashleftrightarrow}[2][]{\ext@arrow 3359\leftrightarrowfill@@{#1}{#2}}
\def\rightarrowfill@@{\arrowfill@@\relax\relbar\rightarrow}
\def\leftarrowfill@@{\arrowfill@@\leftarrow\relbar\relax}
\def\leftrightarrowfill@@{\arrowfill@@\leftarrow\relbar\rightarrow}
\def\arrowfill@@#1#2#3#4{%
  $\m@th\thickmuskip0mu\medmuskip\thickmuskip\thinmuskip\thickmuskip
   \relax#4#1
   \xleaders\hbox{$#4#2$}\hfill
   #3$%
}
\makeatother

\begin{document}

\title[CM algebraic tori]{Class numbers of CM algebraic tori, CM abelian varieties and components of unitary Shimura varieties}

\author{Jia-Wei Guo}
\address{(Guo) Department of Mathematics, National Taiwan
University, No.~1, Roosevelt Rd. Sec.~4  
Taipei, Taiwan, 10617}
\email{jiaweiguo312@gmail.com}


\author{Nai-Heng Sheu}
\address{(Sheu) Department of Mathematics, Indiana University, Rawles
  Hall, 831 East 3rd St. Bloomington, IN, USA, 47405}
\email{naihsheu@iu.edu}


\author{Chia-Fu Yu}
\address{(Yu) Institute of Mathematics, Academia Sinica and NCTS,
6F Astronomy Mathematics Building, No.~1, Roosevelt Rd. Sec.~4  
Taipei, Taiwan, 10617}
\email{chiafu@math.sinica.edu.tw}


\date{\today}

\maketitle

\begin{abstract}
    We give a formula for the class number of an arbitrary CM algebraic torus over $\Q$. This is proved based on results of Ono and Shyr. As applications, we give formulas for numbers of polarized CM abelian varieties, of connected components of unitary Shimura varieties and of certain polarized abelian varieties over finite fields. We also give a second proof of our main result.   
\end{abstract}

\section{Introduction}

An algebraic torus $T$ over a number field $k$ is a connected linear algebraic group over $k$ such that $T\otimes_k \bar k$ isomorphic to 
$(\Gm)^d \otimes_k \bar k$ over the algebraic closure $\bar k$ of $k$ for some integer $d\ge 1$. The class number, $h(T)$, of $T$ is by definition, the cardinality of $T(k)\backslash T(\A_{k,f})/U_T$, where $\A_{k,f}$ is the finite adele ring of $k$ and $U_T$ is the maximal open compact subgroup of $T(\A_{k,f})$.
As a natural generalization for the class number of a number field, Takashi Ono \cite{Ono-arithmetic-of-tori, Ono-Tamagawa-of-tori-Ann} studied the class numbers of algebraic tori. Let $K/k$ be a finite extension and let $R_{K/k}$ denote the Weil restriction of scalars form $K$ to $k$, then we have the following  exact sequence of tori defined over $k$
$$
1\longrightarrow R^{(1)}_{K/k}(\G_{{\rm m},K})\longrightarrow R_{K/k}(\G_{{\rm m},K})\longrightarrow \G_{{\rm m},k}\longrightarrow 1,
$$
where $R^{(1)}_{K/k}(\G_{{\rm m},K})$ is the kernel of the norm map $N: R_{K/k}(\G_{{\rm m},K})\longrightarrow\G_{{\rm m},k}$. It is easy to see that $h(R_{K/k}(\G_{{\rm m},K}))$ and $h(\G_{{\rm m},k})$ coincide with the class numbers $h_K$ and $h_k$ of $K$ and $k$, respectively. In order to compute the class number $h(R^{(1)}_{K/k}(\G_{{\rm m},K}))$, Ono \cite{ono:nagoya1987} introduced the arithmetic invariant 
\[ E(K/k):= \frac{h_K}{h_k\cdot h(R^{(1)}_{K/k}(\G_{{\rm m},K}))} \]
and expressed it in terms of certain cohomological invariants when $K/k$ is Galois. In \cite{katayama:kyoto1991}, S.~Katayama proved a formula for $E(K/k)$ for any finite extension $K/k$. He also studied its dual arithmetic invariant $E'(K/k)$ and gave a similar formula. The latter gives a formula for the class number of the quotient torus $R_{K/k}(\G_{{\rm m},K})/\G_{{\rm m},k}$. The class numbers of general tori have been investigated by J.-M.~Shyr \cite[Theorem 1] {Shyr-class-number-relation}, M.~Morishita \cite{morishita:nagoya1991}, C.~Gonz\'{a}lez-Avil\'{e}s \cite{gonzalez:mrl2008,gonzalez:crelle2010} and M.-H.~Tran \cite{tran:jnt2017}.

Besides the class number $h(T)$ of an algebraic torus $T$, 
another important arithmetic invariant is the Tamagawa number $\tau(T)$.
Roughly speaking, the Tamagawa number is the volume of a suitable fundamental domain. 
More precisely, for any connected semi-simple algebraic group $G$ over
$k$, one associates the group $G(\A_k)$ of adelic points on $G$, where
$\A_k$ is the adele ring of $k$. As $G(\A_k)$ is a unimodular locally
compact group, it admits a unique Haar measure up to a
scalar. T.~Tamagawa defined a canonical Haar measure on $G(\A_k)$ now
called the Tamagawa measure. The Tamagawa number $\tau(G)$ is then
defined as the volume of the quotient space $G(k)\backslash
  G(\A_{k})$ (or a fundamental domain of it) with respect to the Tamagawa measure. 
Similar to the case of class numbers, the calculation of the Tamagawa number is usually difficult. 
A celebrated conjecture of Weil states that any semi-simple simply connected algebraic group has Tamagawa number $1$. The Weil conjecture has been proved in many cases by many people (Weil, T.~Ono, Langlands, K.-F. Lai and others) and 
it is finally proved by Kottwitz \cite{Kottwitz-Tamagawa-numbers}. 


For a more general linear algebraic group $G$, the quotient
  space $G(k)\backslash G(\A_k)$ may not have finite volume. This
occurs precisely when $G$ has non-trivial characters defined over $k$, that is also the case for tori. For this reason, the necessity of introducing convergence factors in a canonical way leads to the emergence of Artin $L$-functions. We shall recall the definition of the Tamagawa number $\tau(T)$ for any algebraic torus $T$.
Then the famous analytic class number formula can be reformulated by the statement $\tau(\G_{{\rm m},k})=1$. 

 
In this paper we investigate the class numbers of CM tori.
Let $K=\prod_{i=1}^r K_i$ be a CM algebra, where each $K_i$ is a CM field. 
The subalgebra $K^+$ of elements in $K$ fixed by the canonical
involution is the product $K^+=\prod_{i=1}^rK_i^+$ of the maximal
totally real subfield $K_i^+$ of $K_i$.  
Denote $N_i$ as the norm map from $K_i$ to $K^+_i$ and $N_{K/K^+}=\prod_{i=1}^r N_i:K\to K^+$ the norm map. 

Now, we put
$T^K=\prod_{i=1}^r T^{K_i}$ with $T^{K_i}=R_{K_i/\Q}(\G_{{\rm
    m},K_i})$, and $T^{K^+_i}=R_{K^+_i/\Q}(\G_{{\rm m},K_i})$. 
We denote
$$h_K=h(K):=\prod^r_{i=1}h({K_i}),\ \ \ 
h_{K^+}=h({K^+}):=\prod^r_{i=1}h({K^+_i}),\ \ \ Q=Q_{K}:=\prod^r_{i=1} Q_i,$$ 
where $Q_i=Q_{K_i}:=
[O^\times_{K_i}:\mu_{K_i}O^\times_{K^+_i}]$ is the Hasse unit index of
the CM extension $K_i/K_i^+$ and $\mu_{K_i}$ is the torsion subgroup
of $O^{\times}_{K_i}$. One has $h(T^K)=h(K)$ and
$h(T^{K^+})=h(K^+)$. It is known that $Q_i\in \{1,2\}$.
Finally, we let $t=\sum^r_{i=1}t_i$ where $t_i$ is the number of primes in $K^+_i$ ramified in $K_i$.
Then we have the following exact sequence of algebraic tori defined over $\Q$
\begin{equation}\label{exact sequence of T^K}
\begin{tikzcd}
    0\arrow{r} & T^K_1\arrow{r} & T^K\arrow{r}{N_{K\slash K^+}} & T^{K^+}\arrow{r} & 0,
\end{tikzcd}
\end{equation}
where
$T^K_1:=\mathrm{ker} (N_{K/K^+})$, which is the product of norm one subtori $T^{K_i}_1:=\left\{x\in T^{K_i}\mid N_i(x)=1\right\}$.
We regard $\Gm$ as a $\Q$-subtorus of $T^{K^+}$ via the diagonal embedding. Let $T^{K,\Q}$ denote the preimage of $\Gm$ in $T^K$ under the map $N_{K\slash K^+}$. We have the second exact sequence of algebraic tori over $\Q$ as follows. Here, for brevity, we write $N$ for $N_{K\slash K^+}$.

\begin{equation}\label{diagram for exact sequence of T^K}
\begin{tikzcd}
\qquad 0\arrow{r} & T^K_1\arrow{r}{\iota}\arrow[equal]{d} & T^K\arrow{r}{N}& T^{K^+}\arrow{r} & 0 \\
\qquad 0\arrow{r} & T^K_1\arrow{r}{\iota} & T^{K,\Q}\arrow{r}{N}\arrow[hook]{u} & \G_{m}\arrow[hook]{u}{}\arrow{r} & 0.
\end{tikzcd}
\end{equation}
The purpose of this paper is concerned with the class number and the Tamagawa number of $T^{K,\Q}$. Let $T(\Zp)$ denote the unique maximal open compact subgroup of $T(\Qp)$.

\begin{thm}\label{Main theorem}
Let $T^{K,\Q}$ denote the preimage of $\Gm$ in $T^K$ under the map $N_{K\slash K^+}$ as in $(\ref{diagram for exact sequence of T^K})$. 
\begin{enumerate}
\item[(1)] We have
$$\tau(T^{K,\Q})=\frac{2^r}{[\Gm(\A): N(T^{K,\Q}(\A))\cdot \Gm(\Q)]},$$where $r$ is the number of components of $K$. 
\item[(2)] We have
\begin{align*}
h(T^{K,\Q})&=\frac{\prod_{p\in S_{K/K^+}} e_{T,p}}{[\Gm(\A): N(T^{K,\Q}(\A))\cdot \Gm(\Q)]}\cdot h(T^K_1)\\
&=\frac{\prod_{p\in S_{K/K^+}} e_{T,p}}{[\A^\times: N(T^{K,\Q}(\A))\cdot \Q^\times]}\cdot \frac{h(K)}{h({K^+})}\cdot\frac{1}{2^{t-r}Q},
\end{align*}
where 
$e_{T, p}:=[\Z^\times_p: N(T^{K,\Q}(\Z_p))]$ 
and $S_{K/K^+}$ is the set of primes $p$ such that there exists a place $v|p$ of $K^+$ ramified in $K$.
\end{enumerate}
\end{thm}

To make the formulas in Theorem~\ref{Main theorem} more explicit, one needs to calculate the indices $e_{T,p}$ and $[\A^\times: N(T^{K,\Q}(\A))\cdot \Q^\times]$. We determine the index $e_{T,p}$ for all primes $p$; the description in the case where $p=2$ requires local norm residue symbols. For the global index $[\A^\times: N(T^{K,\Q}(\A))\cdot \Q^\times]$, we could only compute some special cases including the biquadratic fields and therefore obtain a clean formula for these CM fields. 
For example if $K=\Q(\sqrt{p},\sqrt{-1})$ with prime $p$, then 
\begin{equation}
    h(T^{K,\Q})=\begin{cases}
    1 & \text{if $p=2$;}\\
    h(-p) & \text{if $p\equiv 3 \pmod 4$;}\\
    h(-p)/2 &  \text{if $p\equiv 1 \pmod 4$,}\\
    \end{cases}
\end{equation}
where $h(-p):=h(\Q(\sqrt{-p}))$.
The global index $[\A^\times: N(T^{K,\Q}(\A))\cdot \Q^\times]$ may
serve another invariant which measures the complexity of CM fields and
it requires further investigation. Nevertheless, the indices $e_{T,p}$
and  $[\A^\times: N(T^{K,\Q}(\A))\cdot \Q^\times]$ are all powers of
$2$ (in fact $e_{T,p}\in \{1,2\}$ if $p\neq 2$). Then from
Theorem~\ref{Main theorem} we deduce 
\[ h(T^{K,\Q})=\frac{h_{K}}{h_{K^+}}\frac{2^e}{2^{t-r}\cdot Q}, \] 
where $t$ and $r$ are as in Theorem~\ref{Main theorem},  $e$ is an
integer with $0\le e\le {e(K/K^+,\Q)}$ and $e(K/K^+,\Q)$ is the
invariant defined in \eqref{eKK}. In particular, we conclude
that  $h(T^{K,\Q})$ is equal to $h_{K}/h_{K^+}$ only up to $2$-power.

It is well known that the double coset space $T^{K,\Q}(\Q)\backslash
T^{K,\Q}(\A_f)/T^{K,\Q}(\wh \Z)$ parameterizes CM abelian varieties
with additional structures and conditions. Thus, Theorem~\ref{Main
  theorem} counts such CM abelian varieties and yields a upper bound
for CM points of Siegel modular varieties. There are several
investigations on CM points in the literature which have interesting
applications and we mention a few for the reader's information.
Ullmo and Yafaev \cite{ullmo-yafaev:2015} give a lower bound for Galois orbits of CM points in a Shimura variety. This plays an important role towards the proof of the Andr\'e-Oort conjecture under the Generalized Riemann Hypothesis. Under the same assumption Daw \cite{daw:torsion2012} proves a upper bound of $n$-torsion of the class group of a CM torus, motivated from a conjecture of S.-W. Zhang \cite{swzhang:2005}.    

On the other hand, one can also express the number of connected components of a complex unitary Shimura variety $\Sh_U(G,X)_\C$ as a class number of $T^{K,\Q}$ or $T^K_1$. Thus, our result also gives an explicit formula for $|\pi_0(\Sh_U(G,X)_\C)|$. This information is especially useful when the Shimura variety $\Sh_U(G,X)$ (over the reflex field) has good reduction modulo $p$. Indeed, by the existence of a smooth toroidal compactification due to K.-W. Lan \cite{lan:thesis}, the geometric special fiber $Sh/\Fpbar$ of $\Sh_U(G,X)$ has the same number of connected components of $\Sh_U(G,X)_\C$. In some special cases, one may be able to show that an stratum (eg.~Newton, EO or leaves) in the special fiber is "as irreducible as possible", namely, the intersection with each connected component of $Sh/\Fpbar$ is irreducible. In that case the stratum then has the same number of irreducible components as 
those of $\Sh_U(G,X)_\C$. 

In \cite{achter:GU1n-1} Achter studies the geometry of the reduction modulo a prime p of the unitary Shimura variety associated to $GU(1,n-1)$, extending the work of B\"ultel and Wedhorn \cite{bueltel-wedhorn} (in fact Achter considers one variant of moduli spaces). Though the main result asserts the irreducibility of each non-supersingular Newton stratum in the special fiber $Sh/\Fpbar$, the proof actually shows the "relative irreducibility". That is, every non-supersingular Newton stratum $\calW$ in each connected component of $Sh/\Fpbar$ is irreducible (and non-empty). 
Thus, $\calW$ has $|\pi_0(Sh/\Fpbar)|$ irreducible components and we give an explicit formula for the number of its irreducible components.

There is also a connection of class numbers of CM tori with
the polarized abelian varieties over finite fields. Indeed, 
the set of polarized abelian varieties within a fixed isogeny class can be decomposed into certain orbits which are the analogue of genera of the lattices in a Hermitian space. When the common endomorphism algebra of these abelian varieties is commutative, each orbit is isomorphic to the double coset space associated to either $T^{K,\Q}$ or $T^K_1$ (see Section~\ref{sec:CM}). Marseglia \cite{marseglia:pol_ord_av}
gives an algorithm to compute isomorphism classes of square-free polarized ordinary abelian varieties defined over a finite field. Achter, Altug and Gordon \cite{achter-altug-gordon} also study principally polarized ordinary abelian varieties within an isogeny class over a finite field from a different approach. They utilize the Langlands-Kottwitz counting method and express the number of abelian varieties in terms of discriminants and a product of certain local density factors, reminiscent of the Smith-Minkowski Siegel formula (cf.~\cite[Section 10]{gan-yu:duke2000}).


This paper is organized as follows. Section 2 recalls the definition of the Tamagawa number of an algebraic torus. The proof of Theorem~\ref{Main theorem} is given in Section~\ref{sec:P}. In Section~\ref{sec:I} we compute the local and global indices appearing in Theorem~\ref{Main theorem} and give an improvement and a second proof. We calculate the class number of the CM torus associated to any biquadratic CM field in Section~\ref{sec:E}. In the last section we discuss applications of Theorem~\ref{Main theorem} to polarized CM abelian varieties, connected components of unitary Shimura varieties and polarized abelian varieties with commutative endomorphism algebras over finite fields.

\section{Tamagawa numbers of algebraic tori}
Following \cite{Ono-arithmetic-of-tori}, we recall the definition of
Tamagawa number of an algebraic torus $T$ over a number field $k$. 
Fix the natural Haar measure $dx_v$ on $k_v$ for each place $v$ 
such that it has measure $1$ on the ring of integers $\o_v$ in the
non-archimedean case,  
measure $1$ on $\R/\Z$ in the real place case, and measure $2$ on
$\C/\Z[i]$ in the complex place case. 
Let $\omega$ be a nonzero invariant differential form of $T$ of
highest degree defined over $k$. 
To each place $v$, one associates a Haar measure $\omega_v$ on
$T(k_v)$. 
We say that the product of the Haar measures 
\begin{equation}\label{product of local measure}
\omega_{\A}={\prod_v} \omega_v
\end{equation}
converges absolutely if the product
\[ \prod_{v\nmid \infty} \omega_v(T(\o_v)) \]
converges absolutely, where $T(\o_v)\subset T(k_v)$ 
is the maximal open compact subgroup. In this case, one defines
a Haar measure  $\omega_\A$ on the locally compact topological group
$T(\A_k)$. 
Since the space of invariant differential forms is a one-dimensional
$k$-vector space, 
by the product formula, 
the Haar measure $\omega_\A$ does not depend on the choice of
$\omega$, which is called the canonical measure. 
 
However, the measure $(\ref{product of local measure})$
does not converge if $T$ admits a non-trivial rational
character. Thus, we must modify the local measures by suitable
convergence factors $\lambda_v$ for each $v$ so that  
the product ${\prod\limits_v} (\lambda_v \cdot\omega_v)$ is
absolutely convergent 
on $T(\A_k)$.
Such a collection $\lambda=\left\{\lambda_v\right\}$ is called a set
of convergence factors for $\omega$; the resulting measure is denoted
by $\omega_{\A,\lambda}$. 

Suppose $T$ splits over a Galois extension $K/k$ with Galois group $\grg$.
The group $\wh T:=\Hom_K(T,\Gm)$ of characters is a finite free $\Z$-module with a continuous action of $\grg$. Let $\chi_T:\grg\to \C$ be the character associated to the representation 
$\wh T\otimes \Q$ of $\grg$.

Let $\chi_i,\ 1\leq i\leq h,$ be all the irreducible characters of $\mathfrak{g}$ and we denote $\chi_1$ as the trivial character.
Express
$\chi_T=\sum^h_{i=1}a_i\chi_i$
as the sum of irreducible characters $\chi_i$ with non-negative
integral coefficients. Note that $a_1$ is the rank of the group
  $(\wh T)^\grg$ of rational characters. 
The Artin $L$-function of $\chi_T$ with respect to the field extension $K/k$ is equal to
$$L(s, \chi_T,K\slash k)=\zeta_k(s)^{a_1}\prod_{i=2}^h L(s, \chi_i,K\slash k)^{a_i}.$$ 
We define the number $\rho(T)$ to be the non-zero number
$\lim_{s\rightarrow 1}(s-1)^{a_1}L(s,\chi_T,K/k)$, i.e.,
$$
\rho(T)=\left(\Res _{s=1}\zeta_k(s)\right)^{a_1}\prod_{i=2}^h L(1, \chi_i,K/k)^{a_i}.
$$

On the other hand,
note that there exists a finite set $S$ of places of $k$ such that 
$T\otimes k_v$ admits a smooth model over $\o_v$ 
for each finite place $v$ outside $S$. 
For such $v$, the reduction map 
$
T(\o_v)\rightarrow T(k(v))
$ is surjective, where $k(v)\simeq \F_{q_v}$ is the residue field of $\o_v$.
Let $T^{(1)}(\o_v)$ be the kernel of the reduction map. By \cite[Theorem 2.2.5]{Weil-book} and \cite[(3.3.2)]{Ono-arithmetic-of-tori}, we have
\begin{align*}
\int_{T(\o_v)}\ \omega_{v}&=\vert T(k(v))\vert\times\int_{T^{(1)}(\o_v)}\ \omega_{v} \\
&=\frac{\vert T(k(v))\vert}{q_v^{d}}=L_v(1,\chi_T,K/k)^{-1},
\end{align*}
where $d$ is the dimension of $T$ and
$L_v(s,\chi_T,K/k)$ is the local factor of the Artin $L$-function at $v$.
We now choose the set of convergence factors $\left\{\lambda_v\right\}$ such that
$\lambda_v$ is equal to $1$ if $v$ is archimedean and is equal to $L_v(1, \chi_T, K\slash k)$ otherwise, and hence define a measure $\omega_{\A,\lambda}$ on $T(\A_k)$.

Let $\xi_i, i=1,\dots, a_1$, be a basis of $(\wh T)^\grg$. Define
\[ \xi: T(\A_k) \to \R_{+}^{a_1}, \quad x\mapsto 
(\vert\vert \xi_1(x)||,\dots, \vert \vert\xi_{a_1}(x)\vert\vert), \]
where $\R_+:=\{x>0\in \R\}$. 
Let $T(\A_k)^1$ denote the kernel of $\xi$; one has an isomorphism $T(\A_k)/T(\A_k)^1 \simeq \R^{a_1}_{+} \subset (\R^\X)^{a_1}$. Let $d^\times t:=\prod_{i=1}^{a_1} dt_i/t_i$ be the canonical measure on $\R^{a_1}_{+}$. 
Let $\omega^1_{\A,\lambda}$ be the unique Haar measure on $T(\A_k)^1$ such that 
$\omega_{\A,\lambda}=\omega^1_{\A,\lambda}\cdot d^\times t$, that is, for any measurable function $F$
on $T(\A_k)$ one has
\[ \int_{T(\A_k)/T(\A_k)^1} \int_{T(\A_k)^1} F(xt) \,
\omega^1_{\A,\lambda}\cdot d^\times t=\int_{T(\A_k)} F(x)\,
\omega_{\A,\lambda}. \]
By a well-known theorem of Borel and Harish-Chandra \cite[Theorem
  5.6]{platonov-rapinchuk}, the quotient
  space $T(\A_k)^1/T(k)$ has finite volume with respect to a Haar
  measure. 
The Tamagawa number of $T$ 
is then defined by
\begin{equation}
\tau(T):=\frac{|d_k|^{-\frac{\dim T}{2}}\cdot\int_{T(\A_k)^1/T(k)}\ \omega^1_{\A,\lambda}}{\rho(T)},
\end{equation}
where $d_k$ is the discriminant of the field $k$. 


\section{Proof of Theorem~\ref{Main theorem}}\label{sec:P}

\subsection{$q$-symbols and relative class numbers}
Suppose $\alpha : G \rightarrow G'$ is a homomorphism of abelian groups such that $\ker\alpha$ and $\coker\ \alpha$ are finite. Following Tate, the {\it q-symbol of $\alpha$} is defined by $$q(\alpha):=\vert\coker\ \alpha\vert\slash \vert\ker\alpha\vert .$$ It is easy to see whenever both $G$ and $G'$ are finite, one has $q(\alpha)=\vert G'\vert/\vert G\vert$. 
Let $\Gamma=\Gamma_\Q=\Gal(\Qbar/\Q)$ denote the Galois group of $\Q$.
For any
isogeny $\lambda:T\rightarrow T' $ of algebraic tori defined over $\Q$, we have the following induced maps:  
\begin{align*}
&\hat{\lambda}:\wh{T}'\rightarrow \wh{T},\ \ \  
\hat{\lambda}^\Gamma:(\wh{T}')^\Gamma \rightarrow (\wh{T})^\Gamma,  \\ 
& \lambda^c_p=\lambda_{\Zp}:T(\Z_p)\rightarrow T'(\Z_p), \\
& \lambda_\infty:T(\R) \rightarrow T'(\R), \\
&\lambda_{\Z}: T(\Z) \rightarrow T'(\Z).
\end{align*}
Thus, we have the corresponding $q$-symbols. Note that $T(\Z)=T(\Q)\cap [T(\R) \times \prod_{p<\infty}T(\Z_p)]$.


Shyr \cite{Shyr-class-number-relation} showed that
these $q$-symbols play a role in the connection between the ratios of Tamagawa numbers and class numbers of $T$ and $T'$ as follows.

\begin{thm}\label{thm1.2}
Let $\lambda: T \rightarrow T'$ be an isogeny of algebraic tori defined over $\Q$. Then \[\frac{h(T)}{h({T'})}= \frac{\tau(T)}{\tau({T'})}\cdot\frac{q(\lambda_{\infty})}{q(\lambda_{\Z}) q(\hat{\lambda}^\Gamma)}\cdot \prod_{p<\infty} q(\lambda^c_p).\]
\end{thm}
\begin{proof}
See \cite[Theorem 2]{Shyr-class-number-relation}. \qed
\end{proof}

For any exact sequence
\begin{center} 
\begin{tikzcd}
\notag(E)  \qquad 0\arrow{r} & T'\arrow{r}{\iota} & T\arrow{r}{N}& T''\arrow{r} & 0 
\end{tikzcd}
\end{center}
of algebraic tori defined over $\Q$,
we associate a number to the exact sequence $(E)$ by \cite[Section 4]{Ono-Tamagawa-of-tori-Ann}
\begin{equation}\label{tau(E)}
\tau(E):=\tau(T'')\cdot \tau(T')/ \tau(T).
\end{equation}
\begin{thm}\label{theorem tau (E)}
Let $\mu:T''(\Q)/N(T(\Q))\rightarrow T''(\A)\slash N(T(\A))$ and $\hat{\iota}^\Gamma: {\wh{T}}^\Gamma \rightarrow \wh{T'}^\Gamma$ be the maps derived from the exact sequence $(E)$. Then the subgroups $\coker \mu$ and $\ker \mu$ are finite, and we have $$\tau(E)=q(\hat{\iota}^\Gamma)q(\mu)=\vert\coker\ \hat{\iota}^\Gamma\vert \cdot\frac{[T''(\A): N(T(\A))\cdot T''(\Q)]}{[N(T(\A))\cap T''(\Q): N(T(\Q))]}.$$
\end{thm}
\begin{proof}
See \cite[Section 4.3 and Theorem 4.2.1]{Ono-Tamagawa-of-tori-Ann}. \qed
\end{proof}

Now we let
$$T=T^{K, \Q},\ \ \ T'=T^K_1,\ \ \ \mathrm{and}\ \ \ T''=\Gm.$$
Since $x^2N(x)^{-1}$ is of norm $1$, the map $\lambda$ defined by
\begin{equation}\label{the definition of the lambda map}
\lambda: T \rightarrow T' \X T'',\ \ \ x\mapsto (x^2N(x)^{-1},N(x))
\end{equation}
is an isogeny.
Applying Theorem \ref{thm1.2} to this $\lambda$ and Theorem \ref{theorem tau (E)} to the exact sequence $(\ref{diagram for exact sequence of T^K})$ together, we have
\begin{align}
\nonumber\frac{h(T)}{h({T'\times T''})}
=&\tau(E)^{-1}\cdot \frac{q(\lambda_\infty)}{q(\lambda_{\Z})
q(\hat{\lambda}^\Gamma)} \prod_p q(\lambda_p^c)\\
 =& \vert\coker\ \hat{\iota}^\Gamma\vert^{-1}\cdot  \left(\frac{[T''(\A): N(T(\A))\cdot T''(\Q)]}{[N(T(\A))\cap T''(\Q): N(T(\Q))]}\right) ^{-1} \\ 
\nonumber & \cdot\frac{q(\lambda_\infty)}{q(\lambda_{\Z})q(\hat{\lambda}^\Gamma)}\cdot \prod_{p<\infty} q(\lambda_p^c). 
\end{align}
As $h(\Gm)=1$, we obtain
\begin{equation}\label{h_T/h_{TxT'}}
\begin{split}
  \frac{h(T^{K,\Q})}{h({T^K_1})}=&  \vert\coker\ \hat{\iota}^\Gamma\vert^{-1}\cdot  \frac{[N(T^{K,\Q}(\A))\cap \Gm(\Q): N(T^{K,\Q}(\Q))]}{[\Gm(\A): N(T^{K,\Q}(\A))\cdot \Gm(\Q)]} \\ & \cdot\frac{q(\lambda_\infty)}{q(\lambda_{\Z})q(\hat{\lambda}^\Gamma)}\cdot \prod_{p<\infty} q(\lambda_p^c). 
\end{split}
\end{equation}


We shall determine each term in \eqref{h_T/h_{TxT'}}.

\subsection{Calculation of cokernel} 

\begin{lemma}\label{4.3}
The cardinality of $\coker \hat{\iota}^\Gamma$ is $1$.
\end{lemma}
\begin{proof}
Taking the character groups of (\ref{diagram for exact sequence of T^K}), we have
\begin{equation}\label{eq: Lemma 4.3}  
    \begin{tikzcd}      
    0\arrow{r} & (\wh{\Gm})^\Gamma \arrow{r}{\wh{N}^\Gamma} & (\wh{T^{K,\Q}})^\Gamma\arrow{r}{\hat{\iota}^\Gamma} & (\wh{T^K_1})^\Gamma\arrow{r} & \coker \hat{\iota}^\Gamma \arrow{r} & 0. 
\end{tikzcd}
\end{equation}
Thus, it suffices to show $(\wh{T^K_1})^\Gamma=0$.
Again from $(\ref{diagram for exact sequence of T^K})$, we have
\begin{equation*}
(*)\quad  \begin{tikzcd}    
    0\arrow{r}& (\wh{T^{K^+}})^\Gamma \arrow{r}{\wh{N_{}}}&  (\wh{T^K})^\Gamma\arrow{r}& (\wh{T^K_1})^\Gamma .         
\end{tikzcd}   
\end{equation*}
Recall that $T^K=\prod_{i=1}^r T^{K_i}$ with $T^{K_i}=R_{K_i/\Q}\G_{{\rm m},K_i}$. 
Let $\Gamma_F:=\Gal(\Qbar/F)$.
Also, note that
\begin{align*}
&\wh{T^{K_i}}=\Z[\Gamma_\Q/\Gamma_{K_i}],\ \ \ \ \ \ \wh{T^{K^+_i}}=\Z[\Gamma_\Q/\Gamma_{K_i^+}],\\
&(\wh{T^{K_i}})^\Gamma=\Z[\sum_{\sigma \in \Gamma_\Q/\Gamma_{K_i}}\sigma],\ \ \ \ \ \
(\wh{T^{K^+_i}})^\Gamma=\Z[\sum_{\sigma \in \Gamma_\Q/\Gamma_{K^+_i}}\sigma].
\end{align*}
The norm map $N_i$ sends $x$ to $x \bar{x}$, where $x \in K_i$ and $\bar{x}$ is the complex conjugate of $x$. Therefore, ${\wh{N_i}}(\chi_i)=\chi_i+ \bar{\chi_i}$
for $\chi_i \in \wh{T^{K_i^+}}$. This shows 
$(\wh{T^{K^+_i}})^\Gamma \stackrel{\sim}{\longrightarrow}(\wh{T^{K_i}})^\Gamma$.
Note that the left exact sequence ($*$)$\otimes_\Z \Q$ is also right exact. 
Thus, $(\wh{T^{K_i}_1})^\Gamma\otimes \Q=0$ and $(\wh{T^{K_i}_1})^\Gamma$ is a torsion $\Z$-module. It follows that $(\wh{T^{K_i}_1})^\Gamma=0$, because it is a submodule of a finite free $\Z$-module  $(\wh{T^{K_i}_1})$.
It follows that 
$|\coker \hat{\iota}^\Gamma|=1$. \qed

\end{proof}

We remark that
Lemma \ref{4.3} also follows from
$H^1(\Gamma,\wh{\mathbb{G}}_m)=\Hom(\Gamma,\Z)=1$.
The proof will be also used in Lemma~\ref{4.8}.

\subsection{Calculation of indices of rational points}

Recall that  for any commutative $\Q$-algebra $R$, the groups of  $R$-points of $T^K$ and $T^{K,\Q}$
  are 
$$T^K(R)=(K \otimes R)^\times \quad \text{and} \quad
  T^{K,\Q}(R)=\left\{a\in(K\otimes R)^\times: N(a)\in
    R^\times\right\},$$ respectively.
For $v \in V_K$, the union of the sets $V_{K_i}$ of places of $K^+_i$ for $1\leq i\leq r$, we put $K_v=(K_i)_v$ if $v\in V_{K_i}$.
 For any prime $p$, let $S_p$ be the set of places of $K^+$ lying over $p$. We have
\begin{equation}\label{TQp}
  T^{K,\Q}(\Qp)=\{((x_v)_{v},x_p)\in \prod_{v\in S_p} K_v^\times \times \Qp^\times\, \mid N(x_v)=x_p\ \forall\, v\},
\end{equation}
and 
\begin{equation}\label{TZp}
T^{K,\Q}(\Zp)=\{((x_v)_{v},x_p)\in \prod_{v\in S_p} O_{K_v}^\times \times \Zp^\times\, \mid N(x_v)=x_p\ \forall\, v\}.
\end{equation}
Note that $x_p$ is uniquely determined by $(x_v)_v$ and we may also represent an element $x$ in $T^{K,\Q}(\Qp)$ by $(x_v)_v$. 
\begin{lemma}\label{2.1}
We have $N(T^{K,\Q}(\Q_p))\cap \Z_p^\times=N(T^{K,\Q}(\Z_p))$ for every prime number $p$. 
\end{lemma}

\begin{proof}
Clearly, $N(T^{K,\Q}(\Q_p))\cap \Z_p^\times\supset N(T^{K,\Q}(\Z_p))$. We must prove the other inclusion $N(T^{K,\Q}(\Q_p))\cap \Z_p^\times\subset N(T^{K,\Q}(\Z_p))$. 
Suppose $x=(x_v)\in T^{K,\Q}(\Qp)\cap N^{-1}(\Zp^\times)$. We will find an element $x'=(x'_v)\in T^{K,\Q}(\Zp)$ such that $N(x)=N(x')$. 

Suppose $v$ is inert or ramified in $K$ and let $w$ be the unique place of $K$ over $v$. 
Then $x_v\in K_v=K_w$ and $N(x_v)\in \Zp^\times$.
We have $$\ord_w N(x_v)=\ord_w (x_v)+\ord_w {\bar x}_v=2 \ord_w x_v=0$$ and hence $x_v\in O_{K_w}^\times$. 

Suppose $v=w \bar w$ splits in $K$. Then $x_v=(x_w, x_{\bar w})\in K_v=K_w\times K_{\bar w}=K^+_v\times K^+_v$ and $N(x_v)=x_w x_{\bar w}\in \Zp^\times$.
We have $\ord_v x_w=-\ord_v x_{\bar w}=a_v$ for some $a_v\in \Z$. Put $x_v':=(\varpi_v^{-a_v} x_w, \varpi_v^{a_v} x_{\bar w})$, where $\varpi_v$ is a uniformizer of $K^+_v$. 
Clearly, $x'_v\in O_{K_v}^\times$ and $N(x_v)=N(x'_v)$.

Now suppose $y\in N(T^{K,\Q}(\Qp))\cap \Z_p^\times$ and $N(x)=y$ for some $x\in T^{K, \Q}(\Qp)$.
Set $x':=(x'_v)$ with $x'_v= x_v$ if $v$ is inert or ramified in $K$, and $x'_v$ as above if $v$ splits in $K$.  Then $y=N(x)=N(x')\in N(T^{K,\Q}(\Z_p))$. This proves $N(T^{K,\Q}(\Q_p))\cap \Z_p^\times\subset N(T^{K,\Q}(\Z_p))$.\qed
\end{proof}

\begin{lemma}\label{tau(E)=N} We have
$[N(T^{K,\Q}(\A))\cap \Gm(\Q): N(T^{K,\Q}(\Q))]=1$ and 
$$
\tau(E)=[\Gm(\A): N(T^{K,\Q}(\A))\cdot \Gm(\Q)].
$$
 
\end{lemma}
\begin{proof}
Since $T^{K,\Q}(\A)=\{x\in \A_K^\X \mid N(x)\in\A^\X\}$, we have $N(T^{K,\Q}(\A))=N(\A_K^\times)\cap \A^\times$.
Applying the norm theorem \cite[Theorem 6.1.1]{Ono-Tamagawa-of-tori-Ann}, we have $N(\A_K^\times)\cap (K^+)^\X=N(K^\X)$.
Hence
$$
N(T^{K,\Q}(\A))\cap \Q^\X=N(\A_K^\times)\cap \Q^\times=N(K^\X)\cap \Q^\times=N(T^{K,\Q}(\Q)).
$$
This proves the first statement. The second statement then follows from Lemma~\ref{4.3} and Theorem~\ref{theorem tau (E)}. \qed
\end{proof}

\subsection{Calculation of $q$-symbols}
We are going to evaluate each $q$-symbol
in $(\ref{h_T/h_{TxT'}})$.
Recall the isogeny in $\eqref{the definition of the lambda map}$, namely
\begin{equation}\label{lambda map}
\lambda: T^{K,\Q} \rightarrow T^K_1 \X \Gm,\ x \mapsto (x^2N(x)^{-1}, N(x)).
\end{equation}
Note that $\ker \lambda=\{x\in T^{K,\Q}\mid N(x)=1,\ x^2=1\}=\{x\in T^K_1|\ x^2=1\}$. 
Hence, we have $\ker \lambda= \ker \Sq_{T^K_1},$ where
$\Sq_{T^K_1}: T^K_1 \rightarrow T^K_1,\  x \mapsto x^2$ is the squared map.  

\begin{lemma}\label{4.6}
Suppose 
$d=[K^+:\Q]$.
The $q$-symbol of $\lambda_\infty$ is equal to $2^{-d+1}$.
\end{lemma}
\begin{proof}
Since $K=\prod_{i=1}^r K_i$ is a CM algebra, we have $T^K(\R)=(K\otimes_\Q \R)^\X=(\C^d)^\X$. According to $(\ref{lambda map})$, we have
\begin{align*}
    T^{K,\Q}(\R)&=\{ (x_i) \in (\C^\X)^d|\ N(x_i)=N(x_j) \in \R,\ \forall i,\ j \}, \\
    T^K_1(\R)&=\{(x_i)\in (\C^\X)^d|\ x_i\bar{x_i}=1 \text{ for all } i\}=(S^1)^d, \\
    \ker \lambda_\infty&=\ker \Sq_{T^K_1(\R)}
=\{\pm 1\}^d,
\end{align*}    and hence the exact sequence \begin{center}
\begin{tikzcd}
    0\arrow{r} & \{\pm1\}^d \arrow{r} & T^{K,\Q}(\R)\arrow{r}{\lambda_\infty} & (S^1)^d\X \R^\X.         
\end{tikzcd}
\end{center}
Since $N(\C^\X)=\R^\X_{+}$ is connected, the image of $\lambda_\infty$ is $(S^1)^d\X \R^\X_+$ and $\vert\coker \lambda_\infty\vert=\vert\R^\X / \R^\X_{+}\vert=2$.
Therefore,
$q(\lambda_\infty)=2/\vert\{\pm 1\}^d\vert=2^{-d+1}.$
\qed
\end{proof}

\begin{lemma}\label{4.7}
The $q$-symbol of $\lambda_{\Z}$ is equal to $2$.
\end{lemma}
\begin{proof}
It is clear that $N(x)=1$ for $x\in T^{K,\Q}(\Z)$. Note that any element $x\in O_K^\times$ with $x\bar x=1$ is a root of unity.
Then
\[ T^{K,\Q}(\Z)=T^K_1(\Z)=\prod_{i=1}^r \mu_{K_i}=:\mu_K,\] 
where $\mu_{K_i}$ is the group of roots of unity in $K_i$.
Since $T^{K,\Q}(\Z)$ and $(T^K_1\X \Gm) (\Z)$ are finite, we have 
$$
q(\lambda_\Z)=\frac{\vert( T^K_1\times \mathbb{G}_m)(\Z)\vert}{\vert
  T^{K,\Q}(\Z)\vert}=\frac{\vert\mu_K\times\left\{\pm1\right\}\vert}
{\vert\mu_K\vert}=2. \text{\qed}
$$
\end{proof}

\begin{lemma}\label{4.8}
The $q$-symbol of $\hat{\lambda}^\Gamma$ is equal  to $1$.
\end{lemma}

\begin{proof}
In the proof of Lemma~\ref{4.3} we have showed that $(\widehat{T^K_1})^\Gamma=0$.
Therefore, the map $\hat{\lambda}:(\widehat{\Gm})^\Gamma\X
(\wh{T^K_1})^\Gamma\rightarrow(\widehat{T^{K,\Q}})^\Gamma$ is just
given by
$\hat{N}:(\widehat{\Gm})^\Gamma\rightarrow(\widehat{T^{K,\Q}})^\Gamma.$
The map $\wh{N}$ is in fact an isomorphism from $\eqref{eq: Lemma
  4.3}$. Therefore, $q(\hat{\lambda}^\Gamma)=1$. \qed 
\end{proof}

\begin{lemma}\label{3.1}
Let $A \xrightarrow{\alpha} B \xrightarrow{\beta} C$ be group homomorphisms of abelian groups with finite $\ker \beta \alpha$ and $\coker \beta \alpha$. Then 
the cardinality of $\coker \beta$ is equal to
\begin{equation}
\begin{split}
[C:\im \beta]&=[C:\im \beta\alpha][\im \beta: \im \beta\alpha]^{-1}\\
&=[C: \im \beta\alpha][B:\im \alpha]^{-1}\cdot
[\ker\beta: \ker \beta\alpha \slash \ker \alpha].
\end{split}
\end{equation}
\end{lemma}
\begin{proof}
The first equality is obvious and
the second equality follows by applying the snake lemma to the following diagram
\begin{center}
\begin{tikzcd}
    0\arrow{r} & \ker \beta\alpha \slash \ker \ \alpha \arrow{r}\arrow[hook]{d} & A \slash \ker \alpha \arrow{r}\arrow[hook]{d}& A \slash \ker \beta\alpha \arrow{r}\arrow[hook]{d} & 0 \\
    0\arrow{r} & \ker \beta \arrow{r} & B \arrow{r} & B \slash \ker \beta \arrow{r} & 0. \quad \text{\qed}
\end{tikzcd} 
\end{center}
\end{proof}

 

\begin{lemma}\label{4.10}
We have
$$q(\lambda_p^c)=\left. \begin{cases} 
e_{T,p} & \text{if } p\neq 2;  \\
    e_{T,2}\cdot 2^d & \text{if } p=2,
\end{cases}\right. 
$$
where $e_{T,p}=[\Z_p^\X: N(T^{K,\Q}(\Z_p))]$ and $d=[K^+:\Q]$.
\end{lemma}
\begin{proof}
By $\eqref{TZp}$, every element $x\in T^{K, \Q}(\Z_p)$ is of the form $((x_v)_{v \in S_p}, x_p)$ with $N(x_v)=x_p$ for all $ v \in S_p$.
Consider the homomorphisms 
$$ (T^K_1\X \Gm)(\Z_p) \xrightarrow{m}T^{K,\Q}(\Z_p) \xrightarrow{\lambda_p^c}(T^K_1\X \Gm)(\Z_p),$$
where $((y_v)_{v\in S_p}, y_p)\xmapsto{m}((y_v y_p)_{v\in S_p},
y_p^2)$ and recall that
  $((x_v)_v,x_p)\xmapsto{\lambda^c_p}((x^2_v)_vx^{-1}_p,x_p)$. It is
easy to see that the composition $\lambda_p^c \circ  m$ is the squared
map $\Sq:(y,y')\mapsto (y^2,(y')^2)$.  
Therefore, by Lemma~\ref{3.1}, we have 
\begin{equation*}\label{coker lambda^c_p}
q(\lambda^c_p)=\frac{\vert\coker \lambda_p^c\vert}{\vert\ker \lambda_p^c\vert}
=\vert\coker \Sq\vert\cdot\vert\coker m]^{-1}\cdot[\ker \Sq:\ker m]^{-1}.
\end{equation*}

First, $\vert\coker \Sq\vert=[T^K_1(\Z_p): T^K_1(\Z_p)^2] \cdot [\Z_p^\X : (\Z_p^\X)^2]$. Suppose $y=((y_v)_v, y_p) \in \ker \Sq$. Then $y_p=\pm 1$, and $y_v^2=1 \ \forall v$. If $v$ is inert or ramified in $K$, then $y_v=\pm 1$. If $v$ splits in $K$, then $y_v=(y_w, y_{\bar{w}})$, and $y_w=\pm 1$, $y_{\bar{w}}= \pm 1$. Since $N(y_v)=1$, $y_v=(1, 1)$ or $(-1, -1)$, i.e., $y_v=\pm 1$. We conclude that $\ker \Sq =\{\pm 1\}^{S_p} \X \{\pm 1\}$.

On the other hand, since $\ker N \subset \im m$, we have $\mathrm{coker}\ m\simeq N(T^{K,\Q}(\Z_p))\slash N(\im m).$ 
Recall that $\im m=\{((x_v x_p)_{v \in S_p}, x_p^2)\}$, one has $N(\im m)=\{x_p^2 \mid x_p \in \Z_p^\X\}$.
Therefore, $\vert\coker m\vert= [N(T^{K,\Q}(\Z_p)): (\Z_p^\times)^2]$. 
Clearly, $\ker m=\{\pm 1\}.$

Now,
\begin{equation}
\begin{split}
q(\lambda^c_p)&=[T^K_1(\Z_p): T^K_1(\Z_p)^2] \cdot [\Z_p^\times : (\Z_p^\times)^2]\cdot
[N(T^{K,\Q}(\Z_p))): (\Z_p^\times)^2]^{-1}\cdot2^{-\vert S_p\vert}\\
&=[T^K_1(\Z_p): T^K_1(\Z_p)^2]\cdot [\Z_p^\X: N(T^{K, \Q}(\Z_p))]\cdot 2^{-\vert S_p\vert}\\
&=[T^K_1(\Z_p): T^K_1(\Z_p)^2]\cdot 2^{-\vert S_p\vert}\cdot e_{T,p}.
\end{split}
\end{equation}
The lemma then follows from Lemma~\ref{4.11}. \qed
\end{proof}

For proving Lemma~\ref{4.11}, we recall the structure theorem 
of $p$-adic local units. 

\begin{prop}\label{localunit}
Let $k/\Qp$ be a finite extension of degree $d$ with ring of integers $O_k$ and residue field $\Fq$.
Then $$k^\X \simeq \Z \oplus \Z/(q-1)\Z\oplus \Z/p^a \Z \oplus \Z_p^d$$ and  $$O_k^\X \simeq \Z/(q-1)\Z\oplus \Z/p^a \Z \oplus \Z_p^d,$$
where $p^a=\vert \mu_{p^\infty}(O_k)\vert .$ 
\end{prop}
\begin{proof}
See \cite[Proposition 5.7, p.~140]{neukirch}.    
\end{proof}







\begin{lemma}\label{4.11}
We have
 \[ [T^K_1(\Z_p): T^K_1(\Z_p)^2]=\left.\begin{cases}
2^{\vert S_p\vert} & \text{if}\ p\neq 2; \\
2^{\vert S_p\vert+d} & \text{if}\ p=2,
\end{cases}\right.\]
and \[ q(\Sq_{T^K_1(\Z_p)})=\left. \begin{cases} 1 & \text{if } p\neq 2; \\
2^d & \text{if } p=2, \end{cases}\right. \]
where $d=[K^+:\Q]$ and $S_p$ is the set of places of $K^+$ lying over $p$.
\end{lemma}

\begin{proof}
Note that $T^K_1(\Z_p)= \prod_{v\in S_p} O_{K_v}^{(1)}$, where $O_{K_v}^{(1)}$ consists of norm one elements in $O_{K_v}^\X$. We need to calculate $[O_{K_v}^{(1)}:(O_{K_v}^{(1)})^2]$ for $v\in S_p$.



Let $d_v=[(K^+)_v:\Q_p]$ and consider the exact sequence $$1\rightarrow O_{K_v}^{(1)}\rightarrow O_{K_v}^\X \rightarrow N(O_{K_v}^\X)\rightarrow1.$$ 
Note that $\rank_{\Z_p}O^\times_{K_{v}}=2d_v$. Since $[O_{K^+_v}^\X :N(O_{K_v}^\X)]$ is finite, we have $\rank_{\Z_p}N(O_{K_v}^\X)=\rank_{\Z_p}O_{K^+_v}^\X=d_v$ and hence $\rank_{\Z_p}O^{(1)}_{K_v}=d_v$.

By Proposition~\ref{localunit}, 
we have $O^{(1)}_{K_v} \simeq A \oplus B \oplus \Z_p^{d_v}$, where $A$ is a finite cyclic group of prime-to-$p$ order and $B$ is a finite cyclic group of $p$-power order.


Suppose $p$ is odd.  Then $O^{(1)}_{K_v}/(O^{(1)}_{K_v})^2 \simeq A/ 2A$.   Since $O^{(1)}_{K_v}$ contains $-1$,  we have $A/2A=\Z/2\Z$. Thus, $[O^{(1)}_{K_v}:(O^{(1)}_{K_v})^2]=2$ and $[T^K_1(\Z_p): T^K_1(\Z_p)^2]= \prod_{v\in S_p}[O^{(1)}_{K_v}:(O^{(1)}_{K_v})^2]=2^{\vert S_p \vert}$.

Suppose $p$ is even. The group $O^{(1)}_{K_v}$ contains $-1$; 
Therefore, $B/2B=\Z/2\Z$. We have $O^{(1)}_{K_v}/(O^{(1)}_{K_v})^2
\simeq \Z/2\Z \oplus (\Z/2 \Z)^{d_v}.$ Therefore,
$[O^{(1)}_{K_v}:(O^{(1)}_{K_v})^2]=2^{1+d_v}$ and $[T^K_1(\Z_p):
T^K_1(\Z_p)^2]= \prod_{v\in
  S_p}[O^{(1)}_{K_v}:(O^{(1)}_{K_v})^2]=2^{\vert S_p\vert +d}$ as
$\sum_{v \in S_p}d_v=d$. This proves the first result. 

The second result follows from the first result and $\vert \ker
\Sq_{T^K_1(\Z_p)}\vert = \vert \{\pm 1\}^{S_p}\vert =2 ^{\vert
  S_p\vert}$ (see the proof of Lemma \ref{4.10}). \qed


\end{proof}

\subsection{Proof of Theorem \ref{Main theorem}}
(1)
By \cite[Remark, p.~128]{ono:tamagawa_no}, we have $\tau(\Gm)=1$,
  and $\tau(T^{K_i}_1)=2$ for each $i$ 
and we conclude 
$$\tau(T^{K,\Q})=\frac{2^r}{[\Gm(\A): N(T^{K,\Q}(\A))\cdot \Gm(\Q)]}$$
from (\ref{tau(E)}), Theorem~\ref{theorem tau (E)} 
and Lemma \ref{tau(E)=N}.

(2) By $(\ref{h_T/h_{TxT'}})$ and Lemmas \ref{4.3}, \ref{tau(E)=N}, \ref{4.6}, \ref{4.7}, \ref{4.8} and \ref{4.10},
we obtain
\begin{equation*}
\begin{split}
 h({T^{K,\Q}}){}&=h({T^K_1})\cdot \frac{1}{[\Gm(\A): N(T^{K,\Q}(\A))\cdot \Gm(\Q)]}\cdot \frac{2^{-d+1}}{2 \cdot 1} \cdot \prod_{p\in S_{K/K^+}} e_{T,p} \cdot 2^d. \\
& =\frac{h({T^K_1})\cdot\prod_{p\in S_{K/K^+}}e_{T,p}}{[\Gm(\A): N(T^{K,\Q}(\A))\cdot \Gm(\Q)]} .
\end{split}  
\end{equation*}
 

It is known (see \cite[(16), p.~375]{Shyr-class-number-relation}) that 
\[ h({T^{K_i}_1})=\frac{h_{K_i}}{h_{K_i^+}} \frac{1}{Q_i\cdot 2^{t_i-1}}, \]
where $Q_i$ is the Hasse unit index of $K_i/K^+_i$, and $t_i$ is the number of primes of $K^+_i$ ramified in $K_i$.
Thus,
\begin{align*}
h(T^K_1)&=\prod^r_{i=1}h(T_1^{K_i})=\prod^r_{i=1}\frac{h_{K_i}}{h_{K_i^+}} \frac{1}{Q_i\cdot 2^{t_i-1}}=\frac{h_{K}}{h_{K^+}} \frac{1}{Q\cdot 2^{t-r}}.
\end{align*}
This completes the proof of the theorem. \qed


\section{Local and global indices}\label{sec:I}

\subsection{Local indices}
Keep the notation of the previous section. Denote by $f_v$ the inertia degree of a finite place $v$ of $K^+$. Let $N_v:=N(O_{K_v}^\times)$ and 
$H_v:=\Zp^\times \cap N_v $, if $v|p$. We define an integer 
$e(K/K^+,\Q, p)$, where $p$ is a prime, as follows. For $p\neq 2$, let 
\begin{equation}
    e(K/K^+,\Q,p):=\begin{cases}
    1 & \text{if there exists $v\in S_p$ ramified in $K$ such that $f_v$ is odd;} \\
    0 & \text{otherwise.}
    \end{cases}
\end{equation}
Suppose $p=2$. For any place $v|2$ of $K^+$, let $\Phi_v:\Z_2^\times \to O_{K^+_v}^\times/(O_{K^+_v}^\times)^2$ be the natural map. Consider the condition
\begin{equation}\label{eq:5.2}
     \Phi_v(\Z_2^\times) \not\subset \ol N_v,
\end{equation}
where $\ol N_v$ is the image of $N_v$ in $O_{K^+_v}^\times/(O_{K^+_v}^\times)^2$.
Note that $[\Z_2^\times:H_v]=2$ if and only if \eqref{eq:5.2} holds (see an explanation in the proof of Proposition \ref{I.1} (4)).
Define 
\begin{equation}
    e(K/K^+,\Q, 2):=\begin{cases}
    0 & \text{if $\Phi_v(\Z_2^\times) \subset \ol N_v$ for all $v|2$;} \\
    2 & \text{if there exist two places $v_1,v_2|2$ of $K^+$ satisfying \eqref{eq:5.2} and $H_{v_1}\neq H_{v_2}$;}  \\
    1 & \text{otherwise.}
    \end{cases}
\end{equation}

Define
\begin{equation}\label{eKK}
    e(K/K^+,\Q):=\sum_{p} e(K/K^+,\Q,p). 
\end{equation}
Note that $e(K/K^+,\Q,p)=0$ if $p\not \in S_{K/K^+}$. Recall that
$e_{T,p}=[\Z^\times_p:N(T^{K,\Q}(\Z_p))]$. We shall evaluate $e_{T,p}$
and interpret $e_{T,p}$ by using the invariant $e(K/K^+,\Q)$ in the
following. 

\begin{prop}\label{I.1} \ 

{\rm (1)} We have $e_{T,p}=[\Zp^\times: \cap_{v\in S_p} H_v]$.

{\rm (2)} We have $e_{T,p}|2$ if $p\neq 2$, and $e_{T,2}|4$.

{\rm (3)} Suppose $p\neq 2$. 
Then $e_{T,p}=2$ if and only if there exists a place $v\in S_p$ such that $v$ is ramified in $K$ and $f_v$ is odd.

{\rm (4)} Suppose $p=2$. 
Then $2|e_{T,p}$ if and only if there exists $v\in S_p$ 
satisfying \eqref{eq:5.2}.
Moreover, $e_{T,p}=4$ if and only if there exist two places $v_1, v_2\in S_p$ satisfying \eqref{eq:5.2} and $H_{v_1}\neq H_{v_2}$.
\end{prop}

\begin{proof}
(1) It follows immediately from the description of $T^{K,\Q}(\Zp)$ that $N(T^{K,\Q}(\Zp))=\cap_{v\in S_p} H_v$.

(2) Consider the inclusion $\Zp^\times/H_v \embed O_{K^+_v}^\times/N(O_{K_v}^\times)$. The latter group has order dividing $2$ by the local norm index theorem. Thus, $H_v\subset \Zp^\times$ is of index $1$ or $2$. In particular $H_v$ contains $(\Zp^\times)^2$. It follows that the intersection $\cap H_v$ contains $(\Zp^\times)^2$, and that $e_{T,p}$ divides $[\Zp^\times:(\Zp^\times)^2]$, which is $4$ or $2$ according as $p=2$ or not.

(3) As $p\neq 2$, we have $e_{T,p}=2$ if and only if there exists $v\in S_p$ such that $[\Zp^\times:H_v]=2$. We must show that 
$[\Zp^\times:H_v]=2$ if and only if $v$ is ramified in $K$ and $f_v$ is odd. 
If $v$ is unramified in $K$, then $N_v:=N(O_{K_v}^\times)$ is equal to $O_{K^+_v}^\times$ and $[\Zp^\times:H_v]=1$. Thus, 
$[\Zp^\times:H_v]=2$ only when $v$ is ramified, and we assume this now.
Note that $N_v\subset O_{K^+_v}^\times$ is the unique subgroup of index $2$ and it contains the principal unit subgroup $1+\pi_v O_{K^+_v}$. Therefore, its image $\ol N_v$ in the residue field $\kappa^\times =\F_{q_v}^\times$ is equal to $(\F_{q_v}^\times)^2$. The inclusion $\Zp^\times \embed O_{K^+_v}^\times$ induces the inclusion $\Fp^\times \embed \F_{q_v}^\times$ for which the image $\ol H_v$ of $H_v$ is equal to $\Fp^\times \cap \ol N_v$. Since $\ol N_v$ is the unique subgroup of $\F_{q_v}^\times$ of index $2$, we have 
\[[\Zp^\times:H_v]=[\Fp^\times:\ol H_v]=[\Fp^\times:\Fp^\times\cap (\F_{q_v}^\times)^2]. \]
Since $\F_{q_v}^\times$ is cyclic, $\Fp^\times \subset (\F_{q_v}^\times)^2$ if and only if $p-1|(q_v-1)/2$. The latter is equivalent to $2|(p^{f_v-1}+\dots +1)$ or that $f_v$ is even. Thus, $[\Zp^\times:H_v]=2$ if and only if $f_v$ is odd. 

(4) Since $N_v$ contains $(O_{K^+_v}^\times)^2$, we have $[\Zp^\times :H_v]=[\Phi_v(\Zp^\times):\Phi_v(\Zp^\times)\cap \ol N_v]$, which is equal to $2$ if and only if $\Phi_v(\Zp^\times) \not\subset \ol N_v$. This proves the first statement. 
Clearly, $e_{T,p}=4$ if and only if there exist two places $v_1,v_2\in S_p$ such that $[\Zp^\times :H_{v_1}]=[\Zp^\times :H_{v_2}]=2$ and $H_{v_1}\neq H_{v_2}$. Then the second statement follows from what we have just proved. \qed
\end{proof}

\begin{cor}\label{5.2}
We have $e_{T,p}=2^{e(K/K^+,\Q,p)}$ for all primes $p$.
\end{cor}

By Proposition~\ref{I.1}, one has $e_{T,2}|4$. We now give an example of $T=T^{K,\Q}$ such that $e_{T,2}=4$. Let 
$E$ and $E'$ be two imaginary quadratic field such that $2$ is ramified in both $E$ and $E'$. We also assume that $N_{E_2/\Q_2}(O_{E_2}^\times)\neq N_{E'_2/\Q_2}(O_{E'_{2}}^\times)$. 
Put $K:=E\times E'$, the product of $E$ and $E'$ (not the composite), 
and then $K^+=\Q\times \Q$. Let $v_1$ and $v_2$ be the two places of $K^+$ over $2$. 
We have $H_{v_1}=N_{E_2/\Q_2}(O_{E_2}^\times)$ and $H_{v_2}=N_{E'_2/\Q_2}(O_{E'_{2}}^\times)$. Then we see that 
$[\Z_2^\times:H_{v_1}]=[\Z_2^\times:H_{v_2}]=2$, and that
by Proposition~\ref{I.1} (1), 
$e_{T,2}=[\Z_2^\times:H_{v_1} \cap H_{v_2}]=4$.


For $f\in \bbN$ and $q=p^f$, denote by $\Q_q$ the unique unramified extension of $\Qp$ of degree $f$ and $\Z_q$ the ring of integers. 

\begin{lemma}\label{unramified} 
Let $f\in \bbN$ and $q=2^f$. 

{\rm (1)} $\Z_q^\X/(\Z_q^\X)^2\simeq (1+2\Z_q)/(1+2\Z_q)^2\simeq (\Z/2\Z)^{f+1}$.

{\rm (2)} We have $1+8\Z_q\subset (1+2\Z_q)^2 \subset 1+4\Z_q$. 
Under the isomorphism 
$(1+4\Z_q)/(1+8\Z_q)\simeq \Fq$ $(\,[1+4a]\mapsto \bar a\,)$, the subgroup $(1+2\Z_q)^2/(1+8\Z_q)$ corresponds to the image $\varphi(\Fq)$ of the Artin-Schreier map $\varphi(x)=x^2-x: \Fq\to \Fq$.

{\rm (3)} We have 
\[ \Z_2^\times \cap (1+2\Z_q)^2=\begin{cases} 
1+4\Z_2 & \text{if $f$ is even;} \\
1+8\Z_2 & \text{otherwise.} \\
\end{cases}\]
\end{lemma}
\begin{proof}
(1) The Teichm\"uller lifting $\omega$ gives a splitting of the exact sequence
$1\to (1+2\Z_q) \to \Z_q^\times \to \Fq^\times \to 1$. Thus,
$\Z_q^\times=\Fq^\times \times (1+2\Z_q)$ and hence
$(\Z_q^\times)^2=\Fq^\times \times (1+2\Z_q)^2$ because $q-1$ is
odd. This proves the first isomorphism. The second isomorphism follows
from $1+2\Z_q\simeq \{\pm 1 \} \times \Z_2^f$; see
Proposition~\ref{localunit}.

(2) For $a\in\Z_q$, we have $(1+2a)^2=1+4(a^2+a)$. Therefore, $(1+2\Z_q)^2 \subset 1+4\Z_q$. On the other hand, for any $b\in \Z_q$, the equation $T^2+T=2b$ has a solution in $\Z_q$ by Hensel's Lemma. This proves $1+8\Z_q\subset (1+2\Z_q)^2$.
The image of $(1+2\Z_q)^2/(1+8\Z_q)$ in $\Fq$ consists of elements $\bar a^2+\bar a=\varphi(\bar a)$ for all $\bar a\in \Fq$.

(3) It is clear that $1+8 \Z_2\subset \Z_2^\X \cap (1+2\Z_q)^2\subset 1+4\Z_2$. Note that $1+4 \Z_2\subset (1+2\Z_q)^2$ if and only if $5\in (1+2\Z_q)^2$. The latter is also equivalent to that the equation $1=t^2-t$ is solvable in $\Fq$ by (2). Since
$\F_4=\F_2[t]/(t^2+t+1)$,the previous condition is the same as 
$\F_4\subset \Fq$, or equivalently, $2|f$. \qed 
\end{proof}

\begin{lemma} Let the notation be as in Lemma~\ref{unramified}.

{\rm (1)} If $f$ is even, then there are $2(2^f-1)$ (resp.~$2^{f+1}$) ramified quadratic field extensions $K/\Q_q$ such that $\Z_2^\times\subset N(O_K^\times)$ (resp. $\Z_2^\times \not\subset N(O_K^\times)$). 

{\rm (2)} If $f$ is odd, then there are $2^{f}-2$ (resp.~$2^{f+2}-2^f$ ) ramified quadratic field extensions $K/\Q_q$ such that $\Z_2^\times \subset N(O_K^\times)$ (resp. $\Z_2^\times \not \subset N(O_K^\times)$). 
\end{lemma}
\begin{proof}
(1) Since $\Q_q^\X/(\Q_q^\X)^2\simeq (\Z/2\Z)^{f+2}$, there are 
$2^{f+2}-1$ subgroups $\wt N\subset \Q_q^\times$ of index $2$. By the local class field theory, there are $2^{f+2}-1$ quadratic extensions $K/\Q_q$, and  $2^{f+2}-2$ of them are ramified. On the other hand, since $\Z_q^\X/(\Z_q^\X)^2\simeq (\Z/2\Z)^{f+1}$, there are $2^{f+1}-1$ subgroups $N\subset \Z_q^\times$ of index $2$. It is not hard to see that for each $N$ there are exactly two ramified extensions $K/\Q_q$ such that $N=N(O_K^\times)$. Suppose $2|f$, then $\Phi(\Z_2^\times)$ is a one-dimensional subspace in $\Z_q^\times/(\Z_q^\times)^2=(\Z/2\Z)^{f+1}$ by Lemma~\ref{unramified} (3). Therefore, there are $2^f-1$ subspaces $\ol N$ of co-dimension one containing $\Phi(\Z_2^\times)$, and $2^f$  subspaces $\ol N$ of co-dimension one not containing $\Phi(\Z_2^\times)$. 

(2) Suppose that $f$ is odd. By Lemma~\ref{unramified} (3) $\Phi(\Z_2^\times)$ is a two-dimensional subspace in $\Z_q^\times/(\Z_q^\times)^2=(\Z/2\Z)^{f+1}$. There are $2^{f-1}-1$ subspaces $\ol N$ of co-dimension one containing $\Phi(\Z_2^\times)$, and the other $2^{f+1}-2^{f-1}$ subspaces $\ol N$ not containing $\Phi(\Z_2^\times)$. This proves the lemma. \qed
\end{proof}

\begin{lemma}\label{5.5}
Let $F/\Q_2$ be a finite extension of $\Q_2$, and let $L/F$ be a quadratic extension of $F$. Then $\Z_2^\times \subset N_{L/F}(O_L^\times)$ if and only if the norm residue symbols $(-1,L/F)=1$ and $(5,L/F)=1$. 
\end{lemma}
\begin{proof}
This follows directly from the basic fact that $\Z_2^\times=\{\pm 1\}\X \ol 
{\<5\>}$, where $\ol {\<5\>}$ is the closure of the cyclic subgroup ${\<5\>}$
in $\Z_2^\times$. \qed
\end{proof}

\subsection{Global indices}
We obtain the following partial results for the global index $[\A^\times: N(T^{K,\Q}(\A))\cdot \Q^\times]$.

\begin{lemma}\label{I.6}
{\rm (1)} We have 
$$[\A^\times: N(T^{K,\Q}(\A))\cdot \Q^\times]=[\wh \Z^\times: \wh \Z^\times \cap N(T^{K,\Q}(\A_f))\cdot \Q_{+}],$$
where $\Q_+:=\Q^\times \cap \R_+$. 

{\rm (2)} The index $[\A^\times: N(T^{K,\Q}(\A))\cdot \Q^\times]$ divides
$\prod_{p\in S_{K/K^+}} e_{T,p}$. 
\end{lemma}
\begin{proof}
(1) We have $\A^\times=\R^\times \times \A_f^\times$ and $N(T^{K,\Q}(\A)) =\R_{+}\times N(T^{K,\Q}(\A_f))$. 
We use $\Q^\times$ to reduce $\R^\times$ to $\R_{+}$. Thus,
\begin{equation}\label{eq:I.5}
    \begin{split}
\frac{\A^\times}{[N(T^{K,\Q}(\A))]\cdot \Q^\times}   
&\simeq \frac{\R^\times\times \A^\times_f}{[\R_{+}\times N(T^{K,\Q}(\A_f))]\cdot \Q^\times}\\
& \simeq\frac{\R_{+}\times \A^\times_f}{[\R_{+}\times N(T^{K,\Q}(\A_f))]\cdot \Q_{+}}\\
& \simeq\frac{\Q_{+}\cdot \wh \Z^\times}{N(T^{K,\Q}(\A_f))\cdot \Q_{+}}\\
& \simeq\frac{\wh \Z^\times}{\wh \Z^\times \cap N(T^{K,\Q}(\A_f))\cdot \Q_{+}}.
    \end{split}
\end{equation}
(2) Since 
$\wh \Z^\times\cap N(T^{K,\Q}(\A_f))\cdot \Q_{+} 
\supset \wh \Z^\times\cap N(T^{K,\Q}(\A_f)),$
the group
$$\A^\times/\left ( N(T^{K,\Q}(\A))\cdot \Q^\times \right )\simeq\wh \Z^\times/\left 
( \wh \Z^\times\cap N(T^{K,\Q}(\A_f))\cdot \Q_{+} \right )$$ is a quotient of $\wh \Z^\times/\left ( \wh \Z^\times\cap N(T^{K,\Q}(\A_f)) \right )$. On the other hand, we have $$\wh \Z^\times\cap N(T^{K,\Q}(\A_f))=\prod_p \left [ \Zp^\times \cap N(T^{K,\Q}(\Qp)) \right ]=
\prod_p N(T^{K,\Q}(\Zp)) $$
by Lemma~\ref{2.1}. Therefore, $[\A^\times: N(T^{K,\Q}(\A))\cdot \Q^\times]$ divides
\[ \prod_p [\Zp^\times: N(T^{K,\Q}(\Zp))]=\prod_{p} e_{T,p}=\prod_{p\in S_{K/K^+}} e_{T,p}.\]
This proves the lemma. \qed
\end{proof}

\begin{lemma}\label{I.7}
Suppose $K$ is a CM field which contains two distinct imaginary quadratic fields $E_1$ and $E_2$. Then $[\A^\times: N(T^{K,\Q}(\A))\cdot \Q^\times]=1$.
\end{lemma}
\begin{proof}
By the global norm index theorem, $[\A^\times: N(\A_{E_i}^\times)\cdot \Q^\times]=2$ for $i=1,2$. Since $E_1\neq E_2$, the subgroup  $N(\A_{E_1}^\times)\cdot N(\A_{E_2}^\times)\cdot \Q^\times$ of $\A^\times$ generated by $N(\A_{E_i}^\times)\cdot \Q^\times$ ($i=1,2$) strictly contains  $N(\A_{E_1}^\times)\cdot \Q^\times$. 
Thus, $[\A^\times: N(\A_{E_1}^\times)\cdot N(\A_{E_2}^\times)\cdot \Q^\times]=1$.
On the other hand, the subgroup $N(T^{K,\Q}(\A))$ contains $N(\A_{E_i}^\times)$ for $i=1,2$. Therefore, $[\A^\times: N(T^{K,\Q}(\A))\cdot \Q^\times]=1$. \qed 
\end{proof}

\subsection{Consequences and a second proof} Using our computation of the local index $e_{T,p}$ and $[\A^\times: N(T^{K,\Q}(\A))\cdot \Q^\times]$,  
we obtain the following improvement of Theorem~\ref{Main theorem}.

\begin{thm}\label{second Main Thm}
Let the notation be as in Theorem~\ref{Main theorem}. We have
\[ h(T^{K,\Q})=\frac{h_{K}}{h_{K^+}}\frac{2^e}{2^{t-r}\cdot Q_{K}}, \] 
where $e$ is an integer with $0\le e\le {e(K/K^+,\Q)}$, 
where $e(K/K^+,\Q)$ is the invariant defined in \eqref{eKK}.
\end{thm}
\begin{proof}
This follows from Theorem~\ref{Main theorem}, Corollary~\ref{5.2} and Lemma~\ref{I.6} (2). \qed
\end{proof}

Since the Hasse unit index $Q_{K}$ is a power of $2$, we obtain the following result from Theorem~\ref{second Main Thm}.

\begin{prop}\label{I.9}
The class number $h(T^{K,\Q})$ is equal to $h_K/h_{K^+}$ up to a power of $2$.
\end{prop}

\npr {\bf A second proof of Theorem~\ref{Main theorem}.} Put $T:=T^{K,\Q}$ and $T':=T^K_1$. We first show that the sequence 
\begin{equation}\label{eq:I.6}
    \begin{CD}
        1 \to T'(\A_f)/[T'(\Q) T'(\wh\Z)] \to  T(\A_f)/[T(\Q) T(\wh\Z)] @>N>> N(T(\A_f))/N[T(\Q) T(\wh\Z)] \to 1.
    \end{CD}
\end{equation}
is exact. The kernel of $N$ is 
\[ T'(\A_f)T(\Q)T(\wh \Z)/[T(\Q)\cdot T(\wh\Z)]\simeq T'(\A_f)/[T'(\A_f)\cap T(\Q)\cdot T(\wh\Z)]. \]  If $t=q u\in T'(\A_f)\cap T(\Q)\cdot T(\wh\Z)$ with $q\in T(\Q)$ and $u\in  T(\wh\Z)$, Then 
\[ N(q)=N(u)^{-1}\in T(\Q)\cap T(\wh \Z)=T(\Z)=T'(\Z).\] 
So $q\in T'(\Q)$ and $u\in T'(\wh \Z)$ and $T'(\A_f)\cap T(\Q)\cdot T(\wh\Z)=T'(\Q)\cdot T'(\wh\Z)$. This proves the exactness of \eqref{eq:I.6}.

We now prove that 
\begin{equation}\label{eq:I.7}
   N(T(\A_f))/N[T(\Q)\cdot T(\wh\Z)]\simeq N(T(\A_f))\cdot \Q_+/[N(T(\wh\Z))\cdot \Q_+]. 
\end{equation}
Suppose $t=q u \in N(T(\A_f))\cap N(T(\wh\Z))\cdot \Q_+$ with $q\in \Q^\X_+$ and $u\in N(T(\wh \Z))$. Then $q$ is a local norm everywhere. Thus, there is an element $x\in K^\times$ such that $N(x)=q$ by the Hasse principle. By the definition the element $x$ lies in $T(\Q)$ and hence $q\in N(T(\Q))$. This verifies \eqref{eq:I.7}. 

Note that
\[ [N(T(\A_f))\cdot \Q_+: N(T(\wh\Z))\cdot \Q_+]=[\A_f^\times:N(T(\wh\Z))\cdot \Q_+]
\cdot [\A_f^\times: N(T(\A_f))\cdot \Q_+]^{-1}, \]  
By \eqref{eq:I.5}, $[\A_f^\times: N(T(\A_f))\cdot \Q_+]=[\A^\times: N(T(\A))\cdot \Q^\times]$. 
It is also easy to see $[\A_f^\times:N(T(\wh\Z))\cdot \Q_+]=\prod_{p} [\Zp^\times: N(T(\Zp))[\A_f^\times:N(T(\wh\Z))\cdot \Q^\times_+]]$. Then Theorem~\ref{Main theorem} follows from \eqref{eq:I.6} and \eqref{eq:I.7}. \qed

\begin{question}

(1) Let $N:\wt T\to T$ be a homomorphism of algebraic tori over $\Q$ such that $T':=\ker N$ is again an algebraic torus. Then by Lang's Theorem, the map $N:\wt T(\A)\to T(\A)$ is open and then 
$[T(\A): N(\wt T(\A))\cdot T(\Q)]$ is finite. What is the index $[T(\A): N(\wt T(\A))\cdot T(\Q)]$? 
When $\wt T=T^K$, $T=T^k$ and $N$ is the norm map, where $K/k$ is a finite extension of number fields, then $[T(\A): N(\wt T(\A))\cdot T(\Q)]=[\A_k^\times:k^\times N(\A_K^\times)]$ is nothing but the global norm index and it is equal to the degree $[K_0:k]$ of the maximal abelian subextension $K_0$ 
of $k$ in $K$ \cite[IX, \S 5, p.~193]{Lang-ANT}. The global norm index theorem requires deep analytic results. It is also expected that one may equally need deep analytic and arithmetic results for computing $[T(\A): N(\wt T(\A))\cdot T(\Q)]$.

(2) Suppose $\lambda:T\to T'$ is an isogeny of tori over $\Q$ of degree $d$. Is it true that for any prime $\ell\nmid d$, the $\ell$-primary parts of $h(T)$ and $h(T')$ are the same? This is inspired by Proposition~\ref{I.9}. 
\end{question}

\section{Examples}\label{sec:E}

\subsection{Imaginary quadratic fields}\label{sec:E.1}
Suppose $K$ is an imaginary quadratic field. Then $K^+=\Q$ and
$T^{K,\Q}=T^K$. Thus, we have $h(T^{K,\Q})=h(T^K)=h_K$ without any computation. 
On the other hand, we use Theorem~\ref{Main theorem} to compute $h(T^{K,\Q})$. 
It is easy to compute that $Q=1$, $e_{T,p}=2$ for each $p\in S_{K/K^+}$ and we have
\begin{align*}
[\Gm(\A): N(T^{K,\Q}(\A))\cdot \Gm(\Q)]&=
[\A^\X :  N_{K/\Q}(\A_{K}^\X)\cdot \Q^\X]=2 
\end{align*}
by the global norm index theorem.
Thus, we have
$$
\tau(T^{K,\Q})=1\ \ \ \mathrm{and}\ \ \ h(T^{K,\Q})/h(T^K_1)=2^{t-1},
$$
where $t$ is the number of rational primes ramified in $K$. 
This also gives the result $h(T^{K,\Q})=h_K$. 

\subsection{Biquadratic CM fields}\label{sec:E.2} 
Let $K$ be a biquadratic CM field and $F$ the unique real quadratic
subfield. Write $F=\Q(\sqrt{d})$, where $d>0$ is the unique
square-free positive integer determined by $F$, and $K=EF$, where
$E=\Q(\sqrt{-j})$ for a square-free positive integer $j$. Finite
places of $F$, $E$ and $K$ will be denoted by $v$, $u$ and $w$,
respectively. Recall that $S_{K/F}$ denotes the set of primes $p$ such
that there exists a place $v|p$ of $F$ which is ramified in $K$. 

\begin{lemma}\label{E.1} Let $K=EF=\Q(\sqrt{d},\sqrt{-j})$ be a
  biquadratic CM field over $\Q$. Assume that
  none of primes of $\Q$ is totally ramified in $K$.

{\rm (1)} A prime $p$ lies in $S_{K/F}$ if and only if $p$ is ramified
in $E$ and is unramified in $F$. 

{\rm (2)} If $p$ is ramified in $E$ and splits in $F$, then $e_{T,p}=2$.

{\rm (3)} If $p$ is ramified in $E$ and is inert in $F$, then $e_{T,p}=1$.
\end{lemma}
\begin{proof}
(1) Suppose a prime $p$ is unramified in $E$. Then every place $v|p$
of $F$  
remains unramified in $K$ (see Lang~\cite[Chap.~II, Sec.~4, Prop.~8
(ii)]{Lang-ANT}) and hence $p\not\in S_{K/F}$. 

Suppose a prime $p$ is both ramified in $E$ and in $F$. Then the
unique place $v|p$ must be unramified in $K$, because if $v$ is
ramified in $K$ then $p$ is totally ramified in $K$ which contradicts
to our assumption. 
Thus, $p$ lies in $S_{K/F}$ if and only if it is ramified in $E$ and
is unramified in $F$. 

(2) Let $v_1,v_2$ be the places of $F$ over $p$. One has
    $F_{v_i}=\Q_p$, $K_p=E_u\times E_u$ and
    $H_{v_i}=N(O_{E_u}^\times)$ for $i=1,2$. Thus,
    $e_{T,p}=[\Z^\times_p:N(O^\times_{E_u})]=2$.

(3) Suppose first that $p\neq 2$. We have $F_v=\Q_{p^2}$ with inertia
    degree $f=2$. Then $e_{T,p}=1$ follows from Proposition~\ref{I.1}
    (3). Now assume $p=2$. By  Lemma~\ref{unramified} (3), one has
    $5\in N(O_{K_w}^\X)$ because $5\in 1+4\Z_2\subset
    (1+2\Z_4)^2\subset N(O_{K_w}^\X)$. By Lemma ~\ref{E.2}, we also
    have $-1\in N(O_{K_w}^\X)$. Thus, by Lemma~\ref{5.5},
    $\Z_2^\X\subset N(O_{K_w}^\X)$ and we obtain $e_{T,2}=1$. This
    proves the lemma. \qed 
\end{proof}

\def\sfE{\mathsf{E}}
\def\sfL{\mathsf{L}}

\begin{lemma}\label{E.2}
Let $\sfE/\Q_2$ be a ramified quadratic extension of $\Q_2$, and 
let $\sfL=\sfE\cdot \Q_4$ be the composite of $\sfE$ and $\Q_4$. Then $-1\in N_{\sfL/\Q_4}(O_\sfL^\times)$.
\end{lemma}
\begin{proof}
Since $\sfE/\Q_2$ is ramified, we can write $\sfE=\Q_2(\sqrt{d_{\sfE}})$ and $O_\sfE=\Z_2[\sqrt{d_{\sfE}}]$ for some $d_\sfE\in \{3,7,2,6,10,14\}$. Put $j:=-d_\sfE$ and then $j \mod 8\in \{1,5,2,6\}$.

Note that $\Q_4=\Q_2[t]$ with $t^2+t+1=0$.  For each element $a \in \Q_4$, write $a=a_0+a_1 t$ with $a_0,\ a_1 \in \Q_2$.
For $x \in O_\sfL$, we write $x=a+b \sqrt{-j}$ where $a, b \in \Z_4$ and $a=a_0+a_1 t$, $b=b_0+b_1 t$. Then \begin{equation*}
    \begin{split}
        N(x)=&a^2+jb^2=(a_0^2 + 2 a_1 a_0 t+ a_1^2t^2)+j(b_0^2+2 b_0b_1 t +b_1^2 t^2)\\ 
        =&[a_0^2+jb_0^2-(a_1^2+jb_1^2)]+[2a_0a_1+2jb_0b_1-(a_1^2+jb_1^2)]t.
    \end{split}
\end{equation*} 
Hence for $x=a+b\sqrt{-j}$ satisfying $N(x)\in \Z_2$, the element $x$ must satisfy the condition \begin{equation}\label{eq lemma 6.2}
    2a_0a_1+2jb_0b_1-(a_1^2+jb_1^2)=0.
\end{equation}

Observe $1+8\Z_4 \subset (1+2\Z_4)^2 \subset N(O_\sfL ^\X)$. 
If $-1 \in N(O_\sfL^\X)/ (1+8\Z_4) \subset (\Z_4/8 \Z_4)^\X$, then $-1 \in N(O_\sfL^\X)$. We may solve the equation $N(x)\equiv -1 \pmod 8$, and regard $a, b \in \Z_4/8 \Z_4$. For simplicity, let $j\in \{1,\ 5,\ 2,\ 6\}.$


{\bf Case $j$ is even}: Since $j$ is even, by $\eqref{eq lemma 6.2}$, we have $2|a_1$ and write $a_1=2c_1.$ Consider the case $2|b_1$. Then we have $2a_0a_1-a_1^2=4a_0c_1-4c_1^2=0 \pmod 8$. Hence $(a_0,\ c_1) \equiv (1,\ 1)$,  $(0,\ 0)$ or $(1,\ 0) \pmod 2$. 

We have $N(x)=[a_0^2+jb_0^2-(a_1^2+jb_1^2)]=a_0^2+jb_0^2-4c_1^2$. For $j=2$, take $(a_0,\ b_0,\ c_1)\equiv(1,\ 1,\ 1) \pmod 2$; for $j=6$, take $(a_0,\ b_0,\ c_1)\equiv (1,\ 1,\ 0) \pmod 2.$ Then $N(x)=-1$.

{\bf Case $j$ is odd}: Consider the case $2 \nmid a_1$ and $\ 2 \nmid b_1$. Since $j$ is odd, the condition $\eqref{eq lemma 6.2}$ is equivalent to 
\begin{equation}\label{eq2 lemma 6.2 }
    a_0a_1+b_0b_1-(1+j)/2\equiv 0 \pmod 4.
\end{equation}
 By $\eqref{eq2 lemma 6.2 }$, we require $a_0-b_0\equiv 1 \pmod 2$. Suppose $2|a_0$ and $2\nmid b_0$, and write $a_0=2c_0$. Then $N(x)=-1$ gives the equation $a_0^2+jb_0^2-(a_1+jb_1^2)=4c_0^2+j-1-j=-1 \pmod 8.$ Thus, $c_0=2 d_0$ for some $d_0 \in \Z_4/8\Z_4$. Moreover, substituting $a_0=4d_0$ into $\eqref{eq2 lemma 6.2 }$, we have the condition $b_0b_1-(1+j)/2\equiv 0 \pmod 4$. Since $j \in \{1, 5\}$, there exists $b_1$ satisfying this condition. 
Conclusively, if we take $(c_0,\ b_0,\ a_1,\ b_1)\equiv (0,\ 1,\ 1,\ 1) \pmod 2$ and $b_0 b_1=(1+j)/2$, then $N(x)=-1$. \qed
\end{proof}

Let $\zeta_n$ denote a primitive $n$th root of unity.

\begin{lemma}\label{totram}
  Let $\sfL$ be a totally ramified biquadratic field extension of
  $\Qp$. Then 

  {\rm (1)} $p=2$ and $\sfL\simeq \Q(\zeta_8)\otimes \Q_2=\Q_2[t]$ with
  relation $t^4+1=0$;

  {\rm (2)} for any quadratic subextension $\sfE$ of $\sfL$ over
  $\Q_2$, one has $N_{\sfL/\sfE} (O_{\sfL}^\times)\supset
  \Z_2^\times$.
\end{lemma}
\begin{proof}
  (1) By local class field theory~\cite{Lang-ANT},
  $\Z_p^\times/N_{\sfL/\Qp}(O_\sfL^\times) \simeq I_p
  =\Gal(\sfL/\Q_p)\simeq \Z/2\Z\times \Z/2\Z$, where $I_p$ is the
  inertia group of $p$. 
  As $\Zp^\times$ is pro-cyclic for odd prime $p$, this is possible
  only when $p=2$. 
  Clearly, $2$ is totally ramified in $\Q(\zeta_8)$. Thus, it suffices
  to show that $\Q_2$ has only one totally ramified biquadratic
  extension, that is, there is only one subgroup $H\subset
  \Z_2^\times$ satisfying $\Z_2^\times/H=\Z/2\Z \times \Z/2\Z$ by the
  existence theorem. Now $\Z_2^\times \simeq \Z/2\Z \times \Z_2$ and
  one
  easily checks that $H=\{0\}\times 2\Z_2$ is the unique subgroup
  satisfying $\Zp^\times / H\simeq \Z/2\Z \times \Z/2\Z$. This proves
  (1).

  (2) Put ${\sfE}_1=\Q_2(\sqrt{2})=\Q_2[t-t^3]$,
  $\sfE_2=\Q_2(\sqrt{-2})=\Q_2[t+t^3]$ and $\sfE_3=\Q_2(\sqrt{-1})=
  \Q_2[t^2]$. The Galois group $\Gal(\sfL/\Q_2)=\{1,
  \sigma_1,\sigma_2,\sigma_3\}$, where $\sigma_1(t)=t^{-1}$,
  $\sigma_2(t)=t^3$ and $\sigma_3(t)=t^5=-t$. Then $\sfE_i$ is the fixed
  subfield of the element $\sigma_i$ for each $i=1,2,3$. We choose a
  uniformizer $\pi_i$ of $\sfE_i$ as $t-t^3$, $t+t^3$ and
  $t^2-1$ for $i=1,2,3$, respectively. Thus, we have
  $(\sfE,\pi)=(\sfE_i, \pi_i)$ for some $i$.
  Let $x=a + b t +c t^2+ d t^3 \in
  O_\sfL$ with $a,b,c, d\in \Z_2$. It is well-known that every element
  in $1+4 \pi O_{\sfE}$ is a square, and hence
  $N_{\sfL/\sfE}(O_{\sfL}^\times)\supset (O_\sfE^\times)^2 \supset 1+4 \pi
  O_\sfE$. To show $N_{\sfL/\sfE}(O_{\sfL}^\times)\supset
  \Z_2^\times$, it suffices to show that
  the group $N_{\sfL/\sfE}(O_{\sfL}^\times)$ mod $4\pi O_\sfE$ contains
  $1,3,5,7$ mod $8$.  

  For $i=1$, we compute 
\[ N_{\sfL/\sfE_1}(x)=(a^2+b^2+c^2+d^2)+(ab-ad+bc+cd)(t-t^3). \]
Put $(a,b,c,d)=(1,1,0,1)$, one has $N(x) \mod 4\pi_1$ is equal to $3 \mod
8$. Put $(a,b,c,d)=(1,0,2,0)$, one has $N(x) \mod 4\pi_1$ is equal to
$7 \mod 8$. Thus, $N_{\sfL/\sfE_1}(O_{\sfL}^\times)\supset
  \Z_2^\times$,

  For $i=2$, we compute 
\[ N_{\sfL/\sfE_2}(x)=(a^2-b^2+c^2-d^2)+(ab+ad-bc+cd)(t+t^3). \]
Put $(a,b,c,d)=(2,0,1,0)$, one has $N(x) \mod 4\pi_2$ is equal to $5 \mod
8$. Put $(a,b,c,d)=(0,1,1,1)$, one has $N(x) \mod 4\pi_2$ is equal to $7 \mod
8$. Thus, $N_{\sfL/\sfE_2}(O_{\sfL}^\times)\supset
  \Z_2^\times$, 

  For $i=3$, we compute 
\[ N_{\sfL/\sfE_3}(x)=(a^2-b^2-c^2+d^2+2bd+2ac)+(2ac-b^2+d^2)(t^2-1). \]
Put $(a,b,c,d)=(1,1,0,1)$, one has $N(x) \mod 4\pi_3$ is equal to $3 \mod
8$. Put $(a,b,c,d)=(2,0,1,0)$, one has $N(x) \mod 4\pi_3$ is equal to
$7 \mod  8$.Thus, $N_{\sfL/\sfE_3}(O_{\sfL}^\times)\supset
  \Z_2^\times$, \qed
  
\end{proof}

\begin{cor}\label{E.25}
  Let $K=EF$ be a biquadratic CM field. If $p$ is a prime 
  totally ramified in $K$, then $p=2$ and $e_{T,2}=1$.  
\end{cor}

\begin{prop}\label{E.3} 
Let $F$ be a real quadratic field and $E$ an imaginary quadratic
field, and let $K=EF$. Then 
\begin{equation}\label{eq:E.1}
 h(T^{K,\Q})=\frac{h_K}{h_F}\cdot \frac{2^s}{2^{t-1}\cdot Q},   
\end{equation}
where $t$ is the number of places of $F$ ramified in $K$,  $s$ is the number of primes $p$ that are ramified in $E$ and split in $F$, and $Q=Q_{K}$.
\end{prop}
\begin{proof}
By Lemma~\ref{E.1} and Corollary~\ref{E.25}, 
$\prod_{p\in S_{K/F}} e_{T,p}=2^s$. Since $K$ contains two distinct imaginary quadratic fields, by Lemma~\ref{I.7}, we have $[\A^\X: N(T^{K,\Q}(\A))\cdot \Q^\times]=1$. Thus, the formula \eqref{eq:E.1} follows from Theorem~\ref{Main theorem}. \qed
\end{proof}

Note that we may rewrite \eqref{eq:E.1} as 
\begin{equation}\label{eq:E.2}
  h(T^{K,\Q})=\frac{h_K}{h_F}\cdot \frac{1}{2^{|S_{K/F}|-1}\cdot Q_K}.   
\end{equation}
Indeed, suppose we let $m$ be the number of primes $p$ that are
ramified in $E$ and inert in $F$. Then $t=m+2s+\delta$ and
$|S_{K/F}|=m+s+\delta$, where $\delta=1$ if $2$ is totally ramified in
$K$ and $\delta=0$ otherwise. So $t-s=|S_{K/F}|$.

If $K\neq \Q(\sqrt{2},\sqrt{-1})$, then by Herglotz \cite{herglotz}
(cf. \cite[Section 2.10]{xue-yang-yu:num_inv})  
\begin{equation}\label{eq:E.3}
    h_K=Q_K\cdot h_F\cdot h_E\cdot h_{E'}/2,
\end{equation}
where $E'\subset K$ is the other imaginary quadratic field. 
By \eqref{eq:E.2} and \eqref{eq:E.3}, we have
\begin{equation}\label{eq:E.4}
\begin{split}
    h(T^{K,\Q})&=\frac{h_E \cdot h_{E'}\cdot Q_K}{2}\frac{1}{2^{|S_{K/F}|-1} Q_K} \\
        &=\frac{h_E \cdot h_{E'}}{2^{|S_{K/F}|}}, \quad \quad \text{if $K\neq \Q(\sqrt{2},\sqrt{-1})$.}
\end{split}
\end{equation}

For $K=\Q(\sqrt{2},\sqrt{-1})=\Q(\zeta_8)$, it is known that
$h(\Q(\zeta_8))=1$, and $Q_{\Q(\zeta_8)}=1$ as $8$ is a prime power
\cite[Corollary 4.13, p.~39]{Washington-cyclotomic}. Moreover
$S_{K/F}=\{2\}$ and $e_{T,2}=1$ (Corollary~\ref{E.25}). Thus,
\begin{equation}\label{eq:E.5}
    h(T^{K,\Q})=1, \quad \text{for } K=\Q(\sqrt{2},\sqrt{-1}). 
\end{equation}

We specialize to the case where $K=K_j=FE=\Q(\sqrt{p}, \sqrt{-j})$,
where $F=\Q(\sqrt{p})$, $E=\Q(\sqrt{-j})$,  $p$ is a prime and
$j\in\{1,2,3\}$. Note that we have
$\Q(\sqrt{2},\sqrt{-2})=\Q(\sqrt{2},\sqrt{-1})=\Q(\zeta_8)$ and
$\Q(\sqrt{3},\sqrt{-3})=\Q(\sqrt{3},\sqrt{-1})=\Q(\zeta_{12})$. We may
assume that $p\neq 2$ if $j=2$ and $p \neq 3$ if $j=3$. 

The set $S_{K/F}$ is given as follows:
\begin{itemize}
    \item[(i)] $S_{K_1/F}=\{2\}$ if $p \equiv 1 \pmod 4$, and $S_{K_1/F}=\emptyset$ otherwise;
    \item[(ii)] $S_{K_2/F}=\{2\}$ if $p \equiv 1 \pmod 4$, and $S_{K_2/F}=\emptyset$ otherwise;
    \item[(iii)] $S_{K_3/F}=\{3\}$ always (recall $p\neq 3$).
\end{itemize}
Thus, 
\begin{equation}\label{eq:E.6}
   |S_{K/F}|=\begin{cases}
   0 & \text{if $j\in \{1,2\}$ and $p\not\equiv 1\pmod 4$;}\\
   1 & \text{otherwise.}
   \end{cases}
\end{equation}


By \eqref{eq:E.4} and \eqref{eq:E.6}, if $K\neq \Q(\sqrt{2},\sqrt{-1})$, 
we have
\begin{equation}\label{eq:E.7}
    h(T^{K,\Q})=\begin{cases}
   h(\Q(\sqrt{-jp}))   & \text{if $j\in \{1,2\}$ and $p\not\equiv 1\pmod 4$;} \\
   h(\Q(\sqrt{-jp}))/2 & \text{otherwise.}
   \end{cases}
\end{equation}


\begin{remark}\label{E.4}
   Observe from formula \eqref{eq:E.4} that for computing the class
   number $h(T^{K,\Q}) $ or $h(T^{K}_1)$, one  needs not to
   calculate the Hasse unit index $Q_K$. For the case
   where $F=\Q(\sqrt{p})$ with prime $p$, one has 
   $Q_K=2$ if and only if $p\equiv 3\pmod 4$, and either
   $K=\Q(\sqrt{p},\sqrt{-1})$ or $K=\Q(\sqrt{p},\sqrt{-2})$; see
   \cite[Prop.~2.7]{xue-yang-yu:num_inv}. 
\end{remark}

\section{Polarized CM abelian varieties and unitary Shimura varieties}\label{sec:CM}

\subsection{CM points}\label{sec:CM.1}
Let $(K,O_K, V,\psi, \Lambda,h)$ be a PEL-datum, where
\begin{itemize}
    \item $K=\prod_{i=1}^r K_i$ be a product of CM fields $K_i$ with canonical involution $\bar{ }$\ ;
    \item $O_K$ the maximal order of $K$;
    \item $V$ is a free $K$-module of rank one;
    \item $\psi:V\times V\to \Q$ be a non-degenerate alternating
      pairing such that  
\[ \psi(ax,y)=\psi(x,\bar a y), \quad \forall a\in K, \ x,y\in V; \]
    \item $\Lambda$ be an $O_K$-lattice with $\psi(\Lambda,\Lambda)\subset \Z$;
    \item $h:\C\to \End_{K_\R}(V_\R)$ be an $\R$-algebra homomorphism such that
\[ \psi(h(z)x,y)=\psi(x, h(\bar z)y),\quad \text{for}\ z\in \C, \ x,y\in V_\R:=V\otimes_\Q \R, \] and that the pairing $(x,y):=\psi(h(i)x,y)$ is symmetric and positive definite.
\end{itemize}

Let $V_\C=V^{-1,0}\oplus V^{0,-1}$ be the decomposition into
$\C$-subspaces such that $h(z)$ acts by $z$ (resp.~$\bar z$) on
$V^{-1,0}$ (resp. $V^{0,-1}$).  
Let $T=T^{K,\Q}$ and $U\subset T(\A_f)$ be an open compact subgroup. Put $g=\frac{1}{2} \dim_\Q V$. Let $M_{(\Lambda,\psi),U}$ be the set of isomorphism classes of tuples $(A,\lambda,\iota,\bar \eta)_\C$, where
\begin{itemize}
    \item $A$ is a complex abelian variety of dimension $g$;
    \item $\iota:O_K\to \End(A)$ is a ring monomorphism;
    \item $\lambda:A\to A^t$ is an $O_K$-linear polarization, i.e., it satisfies $\lambda \iota(\bar b)=\iota(b)\lambda$ for all $b\in O_K$;
    \item $\bar \eta$ is an $U$-orbit of $O_K\otimes \wh \Z$-linear isomorphisms
    \[ \eta: V\otimes \wh \Z \to T(A):=\prod_\ell T_\ell(A) \]
    preserving the pairing up to a scalar in $\wh \Z^\times$, where $T_\ell(A)$ is the $\ell$-adic Tate module of $A$
\end{itemize}
such that 
\begin{itemize}
    \item [(a)] $\det(b; V^{-1,0})=\det(b; \Lie(A))$ for all $b\in O_K$;
    \item [(b)] there exists a $K$-linear isomorphism 
\begin{equation}\label{eq:CM.1}
    (V,\psi)\simeq (H_1(A,\Q), \<\, ,\>_\lambda)
\end{equation}
    that preserves the pairings up to a scalar in $\Q^\times$, where $\<\, ,\>_\lambda$ is the pairing induced by the polarization $\lambda$.
\end{itemize}
Two members $(A_1,\lambda_1,\iota_1,\bar \eta_1)$ and $(A_2,\lambda_2,\iota_2,\bar \eta_2)$ are said to be isomorphic if there exists an $O_K$-linear isomorphism $\varphi:A_1\to A_2$ such that $\varphi^* \lambda_2=\lambda_1$ and $\varphi_* \bar \eta_1=\bar \eta_2$. 

\begin{lemma}\label{CM.1}
Let $T$ be an algebraic torus over $\Q$, $U\subset T(\A_f)$ an open compact subgroup, and $U_\infty\subset T(\R)$ an open subgroup. Then 
\begin{equation}\label{eq:CM.2}
\begin{split}
[T(\A): T(\Q)U_\infty U]& =\frac{[U_T:U]}{[T(\Z)_\infty: T(\Z)_\infty \cap U]} \cdot [T(\A): T(\Q)U_\infty U_T] \\    
&=\frac{[U_T:U]}{[T(\Z)_\infty: T(\Z)_\infty \cap U]} \cdot [T(\R):T(\Z) U_\infty]\cdot h(T), 
\end{split}
\end{equation}
where $U_T:=T(\wh \Z)$ is the maximal open compact subgroup of $T(\A_f)$ and $T(\Z)_\infty=T(\Z)\cap U_\infty$.
\end{lemma}
\begin{proof}
Let $T(\Q)_\infty=T(\Q)\cap U_\infty$. One has $[T(\A):T(\Q)U_\infty U]=[T(\A_f):T(\Q)_\infty U]$.
Now consider the exact sequence 
\begin{equation}\label{eq:CM.3}
         0  \longrightarrow \frac{U_T}{U\cdot T(\Q)_\infty \cap
           U_T}\simeq \frac{U_T\cdot G(\Q)_\infty}{U\cdot
           T(\Q)_\infty} \longrightarrow  \frac{T(\A_f)}{U\cdot
           T(\Q)_\infty} \longrightarrow   \frac{T(\A_f)}{U_T \cdot
           T(\Q)_\infty} \longrightarrow  1. 
\end{equation}

It is easy to verify $U\cdot T(\Q)_\infty \cap U_T=U \cdot T(\Z)_\infty$. Using the exact sequence \eqref{eq:CM.3} and the following one

\begin{equation}\label{eq:CM.4}
        0  \longrightarrow  \frac{T(\Z)_\infty}{T(\Z)_\infty \cap U} \longrightarrow  \frac{U_T}{U} \longrightarrow   \frac{U_T}{U \cdot T(\Z)_\infty} \longrightarrow  1, 
\end{equation}

we obtain the first equation of \eqref{eq:CM.2}. 

Now we prove the second equality. Consider the exact sequence 

\begin{equation}\label{eq:CM.5}
       0  \longrightarrow  \frac{T(\R)}{U_T \cdot T(\Q) \cdot U_\infty \cap T(\R)} \longrightarrow  \frac{T(\A)}{U_T \cdot T(\Q) \cdot U_\infty} \longrightarrow   \frac{T(\A)}{U_T \cdot T(\Q) \cdot T(\R)} \longrightarrow  1, 
\end{equation}
where the inclusion $T(\Q)\embed T(\A)$ is given by the diagonal map,
the map $T(\R)\embed T(\A)=T(\R)\times T(\A_f)$ sends $t_\infty\mapsto
(t_\infty,1)$, and the intersection $U_T \cdot T(\Q) \cdot U_\infty
\cap T(\R)$ is taken in $T(\A)$. Suppose $utu_\infty=(tu_\infty, tu)$
is an element in $U_T \cdot T(\Q) \cdot U_\infty \cap T(\R)$. Then
$tu=1$ and $t=u^{-1}\in T(\Q)\cap U_T=T(\Z)$. Thus, $U_T \cdot T(\Q)
\cdot U_\infty \cap T(\R)=T(\Z) \cdot U_\infty \subset T(\R)$ and we
have $[T(\A): T(\Q)U_\infty U_T]=[T(\R):T(\Z) U_\infty]\cdot
h(T)$. \qed 
\end{proof}

By \cite[4.11]{deligne:shimura}, the set $M_{(\Lambda,\psi),U}$ is isomorphic to the Shimura set $\Sh_U(T,h)\simeq T(\Q)\backslash T(\A_f)/U$. By Lemma~\ref{CM.1}, we have
\begin{equation}
    |M_{(\Lambda,\psi),U}|=\frac{[T(\wh \Z):U]}{[T(\Z): T(\Z)\cap U]} \cdot h(T)=\frac{[T(\wh \Z):U]}{[\mu_K: \mu_K\cap U]} \cdot h(T),
\end{equation}
where $\mu_K=\prod_{i=1}^r \mu_{K_i}$ and $T=T^{K,\Q}$. Using Theorem~\ref{Main theorem}, we obtain the following result.

\begin{prop}\label{CM.2}
We have
\begin{equation}
   |M_{(\Lambda,\psi),U}|=\frac{[T(\wh \Z):U]}{[\mu_K: \mu_K\cap U]}\cdot \frac{h_{K}}{h_{K^+}} \cdot \frac{1}{Q\cdot 2^{t-r}} \cdot \frac{\prod_{p\in S_{K/K^+}}e_{T,p}}{[\A^\times: N(T(\A))\cdot \Q^\times]},
\end{equation}
where $r,t,Q,S_{K/K^+}$ and $e_{T,p}$ are as in Theorem~\ref{Main theorem}.
\end{prop}

\subsection{Connected components of unitary Shimura varieties}\label{sec:CM.2}
In this subsection we consider a PEL-datum $(K,O_K, V,\psi, \Lambda,h)$, where
\begin{itemize}
    \item $K$ is a CM field with canonical involution $\bar{ }$\ ;
    \item $V$ is a free $K$-module of rank $n>1$;
    \item $O_K,\psi,h$ are as in Section~\ref{sec:CM.1}.
\end{itemize}

Let $G=GU_K(V,\psi)$ be the group of unitary similitudes of $(V,\psi)$. 
The kernel of the multiplier homomorphism $c:G\to \Gm$ 
is the unitary group $U_K(V,\psi)$ associated to $(V,\psi)$. 
Let $X$ be the $G(\R)$-conjugacy class of $h$, and
$U\subset G(\A_f)$ an open compact subgroup. The complex Shimura variety 
associated to the PEL datum is defined by
\begin{equation}
    Sh_U(G,X)_\C:=G(\Q)\backslash X\times G(\A_f)/U. 
\end{equation}
As in Section~\ref{sec:CM.1}, we define $M_{(\Lambda,\psi),U}$ as the moduli space of complex abelian varieties $(A,\lambda,\iota,\bar \eta)$ with additional structures satisfying the conditions (a) and (b). By \cite[4.11]{deligne:shimura}, one has $Sh_U(G,X)_\C\simeq M_{(\Lambda,\psi),U}$; this provides the modular interpretation of the Shimura variety $Sh_U(G,X)_\C$. We are interested in the number of the connected components of the moduli space $M_{(\Lambda,\psi),U}$, or equivalently, those of the Shimura variety $\Sh_U(G,X)_\C$.
 
Let $X^+$ be the connected component of $X$ that contains the base point $h$, and let $G(\R)_+:=\Stab_{G(\R)} X^+$ be the stabilizer of $X^+$ in $G(\R)$. We have 
\begin{equation}\label{eq:CM.9}
    \pi_0(\Sh_U(G,X)_\C) \simeq G(\Q)_+\backslash G(\A_f)/U\simeq G(\Q)\backslash G(\A)/G(\R)_{+} U, 
\end{equation}
where $G(\Q)_+:=G(\Q)\cap G(\R)_+$. Let $G^{\rm der}$ be the derived group of $G$, and let $D:=G/G^{\rm der}$ be the quotient torus. Denote by $\nu:G\to D$ the natural homomorphism. 
Note that the derived group $G^{\rm der}=SU_K(V,\psi)$ is semi-simple and simply connected.

\begin{thm}\label{CM.3}
Assume that $G^{\rm der}(\R)$ is not compact. Then the complex Shimura variety $\Sh_U(G,X)_\C$ has
\begin{equation}
   \frac{[D(\wh \Z):\nu(U)]}{[\mu_K:\mu_K\cap \nu(U)]}
\cdot \frac{h_K}{h_{K^+}}\cdot \frac{1}{2^{t-1} Q_K}\cdot
\begin{cases}
1 & \text{if $n$ is even;} \\
\frac{\prod_{p\in S_{K/K^+}} e_{T,p}}{[\A^\times:N(T^{K,\Q}(\A)\cdot \Q^\times]} & \text{if $n$ is odd,}
\end{cases} 
\end{equation}
connected components, where $t, Q_K$ and $e_{T,p}$ are as in Theorem~\ref{Main theorem}.
\end{thm}
\begin{proof}
Using the strong approximation argument and Kneser's theorem (namely, $H^1(\Qp, G^{\rm der})=1$ for all primes $p$), the morphism $\nu$ induces a bijection  
\begin{equation}\label{eq:CM.10}
    \nu: G(\Q)_+\backslash G(\A_f)/U \isoto \nu(G(\Q)_+)\backslash D(\A_f)/\nu(U)
\end{equation}
(see \cite[Lemma 2.2]{yu:components}). 
By \cite[Section 7, p.~393--394]{kottwitz:jams1992}, one has 
\begin{equation}
    D\simeq \begin{cases}
    T^{K,1}\times \Gm & \text{if $n$ is even;} \\
    T^{K,\Q}  & \text{if $n$ is odd.} \\
    \end{cases}
\end{equation}
Using the Hasse principle, one shows that $\nu(G(\Q)_{+})=D(\Q)\cap \nu(G(\R)_+)$. One directly checks 
\begin{equation}
     \nu(G(\R)_+) \simeq \begin{cases}
    T^{K,1}(\R)\times \R_{+} & \text{if $n$ is even;} \\
    T^{K,\Q}(\R)  & \text{if $n$ is odd.} \\
    \end{cases}
\end{equation}
As a result, the intersection $D(\Z)_\infty:=D(\Z)\cap \nu(G(\R)_+)$ is equal to $\mu_K$ for all $n$. 
Applying Lemma~\ref{CM.1}, \eqref{eq:CM.9}, \eqref{eq:CM.10} and the formula for $h(D)$ using Theorem~\ref{Main theorem} we obtain the result. \qed
\end{proof}


\subsection{Polarized abelian varieties over finite fields}

In this subsection we formulate two counting problems for polarized abelian 
varieties over finite fields in an isogeny class and compute their cardinality using 
the class number formula of CM tori. Let $k$ be a finite field.

\def\Isog{{\rm Isog}}

\begin{defn}
Let $\ul A_1=(A_1,\lambda_1)$ and $\ul A_2=(A_2,\lambda_2)$ be two polarized abelian varieties over $k$. 

(1) They ($\ul A_1$ and $\ul A_2$) are \emph{isomorphic}, denoted $\ul A_1\simeq \ul A_2$, 
if there exists an isomorphism $\alpha:A_1\isoto A_2$ such that
$\alpha^* \lambda_2=\lambda_1$. Similarly, their polarized
$\ell$-divisible groups $\ul A_1[\ell^\infty]$ and $\ul
A_2[\ell^\infty]$ are said to be \emph{isomorphic}, denoted  $\ul
A_1[\ell^\infty]\simeq \ul A_2[\ell^\infty]$, if there exists an isomorphism $\alpha_\ell:A_1[\ell^\infty]\isoto A_2[\ell^\infty]$ such that
$\alpha_\ell^* \lambda_2=\lambda_1$.    

(2) They are said to be \emph{in the same isogeny class} if there exists a quasi-isogeny 
$\alpha: A_1 \to A_2$ (i.e., a multiple of $\alpha$ by an integer is an isogeny) such that $\alpha^* \lambda_2=\lambda_1$. Denote by 
$\Isog(A_1,\lambda_1)$ the set of isomorphism classes of
$(A_2,\lambda_2)$ lying in the same isogeny class of
$(A_1,\lambda_1)$.    

(3) They are said to be \emph{similar}, denoted $\ul A_1\sim \ul A_2$,
if there exists an isomorphism $\alpha: A_1 \to A_2$ such that
$\alpha^* \lambda_2=q \lambda_1$ for some $q\in \Q_{>0}$. Similarly, their polarized
$\ell$-divisible groups $\ul A_1[\ell^\infty]$ and $\ul
A_2[\ell^\infty]$ are said to be \emph{similar}, denoted  $\ul
A_1[\ell^\infty]\sim \ul A_2[\ell^\infty]$, if there exists an
isomorphism $\alpha_\ell:A_1[\ell^\infty]\isoto A_2[\ell^\infty]$ such
that 
$\alpha_\ell^* \lambda_2=q \lambda_1$, for some $q\in \Q_\ell^\times$.

(4) They are said to be \emph{isomorphic locally everywhere},
if $$\ul A_1[\ell^\infty]\simeq \ul A_2[\ell^\infty]\quad \text{over
  $k$}$$  
for all primes $\ell$ including the prime $\char  k$.

(5) They are said to be \emph{similar locally everywhere},
if $$\ul A_1[\ell^\infty]\sim \ul A_2[\ell^\infty]\quad \text{over $k$}$$ 
for all primes $\ell$ (also including the prime $\char  k$).

\end{defn} 

Now we start with a polarized abelian variety $(A_0,\lambda_0)$ over $k$. Assume that the endomorphism algebra $\End^0(A_0)$ is commutative. So $\End^0(A_0)=K$ for some CM algebra $K$, and $R:=\End(A_0)\subset K$ is a CM order. Let $T:=T^{K,\Q}$ and $T^1:=T^{K}_1$. On the other hand, consider 
\begin{equation}
\Lambda(A_0,\lambda_0):=\left \{(A,\lambda)\in \Isog(A_0,\lambda_0)\,
  {\Bigg |}\, \text{\parbox{1.8in}{$(A,\lambda)$ is locally isomorphic
      to $(A_0,\lambda_0)$ everywhere}} \right \} 
\end{equation}
and 
\begin{equation}
I(A_0,\lambda_0):=\left \{ \text{\parbox{1.4in}{similitude classes of
      $(A,\lambda)\in \Isog(A_0,\lambda_0)$}}\, {\Bigg |}\,
  \text{\parbox{1.5in}{$(A,\lambda)$ is locally similar to
      $(A_0,\lambda_0)$ everywhere}} \right \}.     
\end{equation}

\begin{prop}\label{CM.5} Let $U:=T(\A_f)\cap \wh R^\times$ and $U^1:=T^1(\A_f)\cap \wh R^\times$. We have 
\begin{equation}\label{eq:CM.13}
    |\Lambda(A_0,\lambda_0)|=\frac{[T^1(\wh \Z):U^1]}{[\mu_K:\mu_K\cap U^1]} \cdot \frac{h_K}{h_{K^+}} \cdot \frac{1}{2^{t-r} Q_K},
\end{equation}
and 
\begin{equation}\label{eq:CM.14}
    |I(A_0,\lambda_0)|=\frac{[T(\wh \Z):U]}{[\mu_K:\mu_K\cap U]}\cdot \frac{h_K}{h_{K^+}} \cdot \frac{1}{2^{t-r} Q_K} \cdot \frac{\prod_{p\in S_{K/K^+}} e_{T,p}}{[\A^\times: N(T(\A))\cdot \Q^\times]}.
\end{equation}
\end{prop}
\begin{proof}
The main point is the following natural bijections
\begin{equation}\label{eq:CM.15}
   \Lambda(A_0,\lambda_0)\simeq T^1(\Q)\backslash T^1(\A_f)/U^1, \quad
   \text{and} \quad I(A_0,\lambda_0)\simeq T(\Q)\backslash T(\A_f)/U.  
\end{equation}
See \cite[Theorem 5.8 and Section 5.4]{xue-yu:counting}; also see
\cite[Theorem 2.2]{yu:smf}.  
Then formulas \eqref{eq:CM.13} and \eqref{eq:CM.14} follow from
\eqref{eq:CM.15}, Theorem~\ref{Main theorem}, and
Lemma~\ref{CM.1}. \qed 
\end{proof}

\section*{Acknowledgments}
Guo is partially supported by the MoST grants 
106-2115-M-002-009MY3.
Yu is partially supported by the MoST grants 
104-2115-M-001-001MY3, 107-2115-M-001-001-MY2 and 109-2115-M-001-002-MY3.
The authors thank the referee for a careful reading and helpful
comments which improve the exposition of this paper.

\bibliographystyle{plain}
\bibliography{TeXBiB}

\def\cprime{$'$}
\begin{thebibliography}{10}

\bibitem{achter-altug-gordon}
Jeff {Achter}, Salim~Ali {Altug}, Julia {Gordon}, Wen-Wei {Li}, and Thomas
  {R{\"u}d}.
\newblock {Counting abelian varieties over finite fields via Frobenius
  densities}.
\newblock {\em arXiv e-prints}, page arXiv:1905.11603, May 2019.

\bibitem{achter:GU1n-1}
Jeffrey~D. Achter.
\newblock Irreducibility of {N}ewton strata in {${\rm GU}(1,n-1)$} {S}himura
  varieties.
\newblock {\em Proc. Amer. Math. Soc. Ser. B}, 1:79--88, 2014.

\bibitem{bueltel-wedhorn}
Oliver B\"{u}ltel and Torsten Wedhorn.
\newblock Congruence relations for {S}himura varieties associated to some
  unitary groups.
\newblock {\em J. Inst. Math. Jussieu}, 5(2):229--261, 2006.

\bibitem{daw:torsion2012}
Christopher Daw.
\newblock On torsion of class groups of {CM} tori.
\newblock {\em Mathematika}, 58(2):305--318, 2012.

\bibitem{deligne:shimura}
Pierre Deligne.
\newblock Travaux de {S}himura.
\newblock In {\em S\'eminaire {B}ourbaki, 23\`eme ann\'ee (1970/71), {E}xp.
  {N}o. 389}, pages 123--165. Lecture Notes in Math., Vol. 244. Springer,
  Berlin, 1971.

\bibitem{gan-yu:duke2000}
Wee~Teck Gan and Jiu-Kang Yu.
\newblock Group schemes and local densities.
\newblock {\em Duke Math. J.}, 105(3):497--524, 2000.

\bibitem{gonzalez:mrl2008}
Cristian~D. Gonz\'{a}lez-Avil\'{e}s.
\newblock Chevalley's ambiguous class number formula for an arbitrary torus.
\newblock {\em Math. Res. Lett.}, 15(6):1149--1165, 2008.

\bibitem{gonzalez:crelle2010}
Cristian~D. Gonz\'{a}lez-Avil\'{e}s.
\newblock On {N}\'{e}ron-{R}aynaud class groups of tori and the capitulation
  problem.
\newblock {\em J. Reine Angew. Math.}, 648:149--182, 2010.

\bibitem{herglotz}
G.~Herglotz.
\newblock \"{U}ber einen {D}irichletschen {S}atz.
\newblock {\em Math. Z.}, 12(1):255--261, 1922.

\bibitem{katayama:kyoto1991}
Shin-ichi Katayama.
\newblock Isogenous tori and the class number formulae.
\newblock {\em J. Math. Kyoto Univ.}, 31(3):679--694, 1991.

\bibitem{Kottwitz-Tamagawa-numbers}
Robert~E. Kottwitz.
\newblock Tamagawa numbers.
\newblock {\em Ann. of Math. (2)}, 127(3):629--646, 1988.

\bibitem{kottwitz:jams1992}
Robert~E. Kottwitz.
\newblock Points on some {S}himura varieties over finite fields.
\newblock {\em J. Amer. Math. Soc.}, 5(2):373--444, 1992.

\bibitem{lan:thesis}
Kai-Wen Lan.
\newblock {\em Arithmetic compactifications of {PEL}-type {S}himura varieties},
  volume~36 of {\em London Mathematical Society Monographs Series}.
\newblock Princeton University Press, Princeton, NJ, 2013.

\bibitem{Lang-ANT}
Serge Lang.
\newblock {\em Algebraic number theory}, volume 110 of {\em Graduate Texts in
  Mathematics}.
\newblock Springer-Verlag, New York, second edition, 1994.

\bibitem{marseglia:pol_ord_av}
Stefano {Marseglia}.
\newblock {Computing square-free polarized abelian varieties over finite
  fields}.
\newblock {\em arXiv e-prints}, page arXiv:1805.10223, May 2018.

\bibitem{morishita:nagoya1991}
Masanori Morishita.
\newblock On {$S$}-class number relations of algebraic tori in {G}alois
  extensions of global fields.
\newblock {\em Nagoya Math. J.}, 124:133--144, 1991.

\bibitem{neukirch}
J{\"u}rgen Neukirch.
\newblock {\em Algebraic number theory}, volume 322 of {\em Grundlehren der
  Mathematischen Wissenschaften}.
\newblock Springer-Verlag, Berlin, 1999.
\newblock Translated from the 1992 German original and with a note by Norbert
  Schappacher.

\bibitem{Ono-arithmetic-of-tori}
Takashi Ono.
\newblock Arithmetic of algebraic tori.
\newblock {\em Ann. of Math. (2)}, 74:101--139, 1961.

\bibitem{Ono-Tamagawa-of-tori-Ann}
Takashi Ono.
\newblock On the {T}amagawa number of algebraic tori.
\newblock {\em Ann. of Math. (2)}, 78:47--73, 1963.

\bibitem{ono:tamagawa_no}
Takashi Ono.
\newblock On {T}amagawa numbers.
\newblock In {\em Algebraic {G}roups and {D}iscontinuous {S}ubgroups ({P}roc.
  {S}ympos. {P}ure {M}ath., {B}oulder, {C}olo., 1965)}, pages 122--132. Amer.
  Math. Soc., Providence, R.I., 1966.

\bibitem{ono:nagoya1987}
Takashi Ono.
\newblock On some class number relations for {G}alois extensions.
\newblock {\em Nagoya Math. J.}, 107:121--133, 1987.

\bibitem{platonov-rapinchuk}
Vladimir Platonov and Andrei Rapinchuk.
\newblock {\em Algebraic groups and number theory}, volume 139 of {\em Pure and
  Applied Mathematics}.
\newblock Academic Press, Inc., Boston, MA, 1994.
\newblock Translated from the 1991 Russian original by Rachel Rowen.

\bibitem{Shyr-class-number-relation}
Jih~Min Shyr.
\newblock On some class number relations of algebraic tori.
\newblock {\em Michigan Math. J.}, 24(3):365--377, 1977.

\bibitem{tran:jnt2017}
Minh-Hoang Tran.
\newblock A formula for the {$S$}-class number of an algebraic torus.
\newblock {\em J. Number Theory}, 181:218--239, 2017.

\bibitem{ullmo-yafaev:2015}
Emmanuel Ullmo and Andrei Yafaev.
\newblock Nombre de classes des tores de multiplication complexe et bornes
  inf\'{e}rieures pour les orbites galoisiennes de points sp\'{e}ciaux.
\newblock {\em Bull. Soc. Math. France}, 143(1):197--228, 2015.

\bibitem{Washington-cyclotomic}
Lawrence~C. Washington.
\newblock {\em Introduction to cyclotomic fields}, volume~83 of {\em Graduate
  Texts in Mathematics}.
\newblock Springer-Verlag, New York, second edition, 1997.

\bibitem{Weil-book}
Andr\'{e} Weil.
\newblock Ad\`eles et groupes alg\'{e}briques.
\newblock In {\em S\'{e}minaire {B}ourbaki, {V}ol. 5}, pages Exp. No. 186,
  249--257. Soc. Math. France, Paris, 1995.

\bibitem{xue-yang-yu:num_inv}
Jiangwei Xue, Tse-Chung Yang, and Chia-Fu Yu.
\newblock Numerical invariants of totally imaginary quadratic {$\Bbb Z[\sqrt
  p]$}-orders.
\newblock {\em Taiwanese J. Math.}, 20(4):723--741, 2016.

\bibitem{xue-yu:counting}
Jiangwei {Xue} and Chia-Fu {Yu}.
\newblock {On counting certain abelian varieties over finite fields}.
\newblock {\em ArXiv e-prints}, January 2018.
\newblock to appear in {Acta Math. Sin. (Engl. Ser.)}.

\bibitem{yu:components}
Chia-Fu Yu.
\newblock Connected components of certain complex {S}himura varieties.
\newblock In preparation.

\bibitem{yu:smf}
Chia-Fu Yu.
\newblock Simple mass formulas on {S}himura varieties of {PEL}-type.
\newblock {\em Forum Math.}, 22(3):565--582, 2010.

\bibitem{swzhang:2005}
Shou-Wu Zhang.
\newblock Equidistribution of {CM}-points on quaternion {S}himura varieties.
\newblock {\em Int. Math. Res. Not.}, (59):3657--3689, 2005.

\end{thebibliography}

\end{document}